\documentclass[11pt]{article}
\usepackage[letterpaper]{geometry}
\usepackage{amsthm,amsmath,amsfonts}
\usepackage{color,graphicx}
\usepackage{amssymb}
\usepackage{mathrsfs}
\usepackage{subfigure}
\RequirePackage[numbers,sort]{natbib}
\RequirePackage[colorlinks,citecolor=blue,urlcolor=blue]{hyperref}
\newcommand{\R}{\mathbb{R}}
\newcommand{\D}{\mathbb{D}}

\newcommand{\E}{\mathbb{E}}
\newcommand{\AC}{\text{AC}}

\newcommand{\Prob}{\mathbb{P}}

\providecommand{\abs}[1]{\lvert#1\rvert}

\newtheorem{theorem}{Theorem}[section]
\newtheorem{lemma}{Lemma}[section]
\newtheorem{proposition}{Proposition}

\theoremstyle{definition}
\numberwithin{equation}{section}

\newtheorem{assumption}{Assumption}

\begin{document}
\noindent

\begin{titlepage}

\vspace{-1.0in}

\title{\bf Optimal Control of Brownian Inventory Models with Convex
  Holding Cost: Average Cost Case \footnote{Research supported
in part by NSF grants CMMI-0727400, CMMI-0825840, and CMMI-1030589}}

\author{J. G. Dai\footnote
{H. Milton Stewart School of Industrial and Systems Engineering,
  Georgia
  Institute of Technology, Atlanta, Georgia 30332, U.S.A.; dai@gatech.edu}
\ and \ Dacheng Yao\footnote
{Academy of Mathematics and Systems Science, Chinese
  Academy of Sciences, Beijing, 100190, China; dachengyao@amss.ac.cn} }
\date{October 14, 2011}
\maketitle
\thispagestyle{empty}
\begin{abstract}
  We consider an inventory system in which inventory level fluctuates
  as a Brownian motion in the absence of control.  The inventory
  continuously accumulates cost at a rate that is a general convex
  function of the inventory level, which can be negative when there is
  a backlog.  At any time, the inventory level can be adjusted by a
  positive or negative amount, which incurs a fixed cost and a
  proportional cost. The challenge is to find an adjustment policy
  that balances the holding cost and adjustment cost to minimize the
  long-run average cost. When both upward and downward fixed costs are
  positive, our model is an impulse control problem.  When both fixed
  costs are zero, our model is a singular or instantaneous control
  problem. For the impulse control problem, we prove that a
  four-parameter control band policy is optimal among all feasible
  policies. For the singular control problem, we prove that a
  two-parameter control band policy is optimal.

  We use a lower-bound approach, widely known as ``the verification
  theorem'', to prove the optimality of a control band policy for both
  the impulse and singular control problems. Our major contribution is
  to prove the existence of a ``smooth'' solution to the free boundary
  problem under some mild assumptions on the holding cost function.
  The existence proof leads naturally to a numerical algorithm to
  compute the optimal control band parameters.
  We demonstrate that the lower-bound approach also works for Brownian inventory
  model in which no inventory backlog is allowed. In a companion
  paper, we will show how   the lower-bound approach can be
  adapted to study a  Brownian inventory model under a discounted
  cost criterion.


\end{abstract}


\medskip 

\end{titlepage}

\section{Introduction}
\label{sec:introduction}
This paper is concerned with optimal control of Brownian inventory
models under the long-run average cost criterion. It serves two
purposes. First, it provides a tutorial on the powerful lower-bound
approach, known as the ``the verification theorem'', to proving the
optimality of a control band policy among all feasible policies. The
tutorial is rigorous and, except the standard It\^o formula, self
contained.  Second, it contributes to the literature by proving the
existence of a ``smooth'' solution to the free boundary problem with a
general convex holding cost function.  The existence proof leads
naturally algorithms to compute the optimal control band parameters.
The companion paper \cite{DaiYao11b} studies the optimal control of
Brownian inventory models under a discounted cost criterion.

\subsection*{The Model Description}
\label{sec:brown-contr-probl}

In this paper and the companion paper \cite{DaiYao11b}, the inventory \emph{netput} process is assumed to follow a
Brownian motion with drift $\mu$ and variance $\sigma^2$. The netput
process captures the difference between regular supplies, possibly
through a long term contract, and customer demands. Controls are
exercised on the netput process to keep the inventory at desired
positions. The controlled process, denoted by $Z=\{Z(t), t\ge0\}$, is
called the \emph{inventory process} in this paper. For each time $t\ge
0$, $Z(t)$ is interpreted as the inventory level at time $t$ although
$Z(t)$ can be negative, in which case $|Z(t)|$ represents the
inventory backlog at time $t$.  We assume that the holding cost
function $h:\R\to \R_+$ is a \emph{general} convex function. Thus, $
\int_0^t h(Z(s))ds$ is the cumulative inventory cost by time $t$.

Inventory position is assumed to be adjustable, either upward or
downward. All adjustments are realized immediately without any
leadtime delay.  Each upward adjustment with amount $\xi>0$ incurs a
cost $K+k\xi$, where $K\ge 0$ and $k>0$ are the fixed cost and the
variable cost, respectively, for each upward adjustment. Similarly,
each downward adjustment with amount $\xi$ incurs a cost of $L+\ell
\xi$ with fixed cost $L\ge 0$ and variable cost $\ell>0$.  The
objective is to find some control policy that balances the inventory
cost and the adjustment cost so that the long-run average total cost
is  minimized.

In describing our Brownian control problems, we have used the
inventory terminology in supply chain management. One could describe
such control problems in cash flow management. In this case, $Z(t)$
represents the cash amount at time $t\ge 0$. There are a large number
of papers in the economics literature that have studied the Brownian
control problems (e.g.  Dixit~\cite{Dixit91}). Readers are referred to
Stokey~\cite{Stokey09} and the references there for a variety of
economic applications of Brownian control problems. While the
discounted cost criterion is appropriate for cash flow management, the
long-run average cost criterion is natural for many
production/inventory problems.

When both fixed costs $K$ and $L$ are positive, it is clear that
non-trivial feasible control policies should limit the number of
adjustments to be finite within any finite time interval.  Under such
a control policy, inventory is adjusted at a sequence of discrete
times and the resulting control problem is termed as the
\emph{impulse} control of a Brownian motion. When both fixed costs
$K=0$ and $L=0$, it can be advantageous for the system to make an
``infinitesimal'' amount of adjustment at any moment. Indeed, as it
will be shown in Section~\ref{sec:instanteneous}, an optimal policy will
make an uncountable number of
adjustments within a finite time interval. The resulting control
problem is termed as the \emph{singular} or \emph{instantaneous}
control of a Brownian motion. In this paper, we treat impulse  and
singular control of a Brownian motion in a single
framework. Conceptually, one may view the singular control problem  as a
limit of a sequence of impulse control problems as fixed costs
$K\downarrow 0$ and $L\downarrow 0$. Such a connection between impulse
and singular control problems allow us to solve a \emph{mixed
impulse-singular} control problem (for example,  $K>0$ and $L=0$)
without much additional effort.

\subsection*{Non-Linear Holding Cost}
\label{sec:non-linear-holding}

When the holding cost function $h$ is given by
\begin{equation}
  \label{eq:linearh}
  h(x) =
  \begin{cases}
    -p x & \text{ if } x <0, \\
    c x & \text{ if } x \ge 0
  \end{cases}
\end{equation}
for some constants $p>0$ and $c>0$, we call $h$ in (\ref{eq:linearh})
a linear holding cost function, even though $h(x)$ in
(\ref{eq:linearh}) is piecewise linear in inventory level $x$.
 With this holding cost
function, inventory backlog cost is linear and inventory excess cost
is also linear, but  $h(x)$ is not differentiable at $x=0$. Although
many papers focused on linear holding cost function
(e.g.~\cite{HarrisonSellkeTaksar83}), there are ample
applications that motivate non-linear holding cost function. For
example, \cite{BuckleyKorn89} and \cite{PliskaSuzuki04} studied optimal
index tracking of a benchmark index when there are transaction
costs.  An impulse control problem with quadratic holding cost arises
naturally in their studies. Quadratic holding cost and general convex
holding cost also arise in
economic papers; see, for example,
\cite{BertolaCabellero90,Tsiddon93,PlehnDujowich05}.

\subsection*{Optimal Policy Structure}

For an impulse Brownian control problem under the long-run average
cost criterion, we prove in Section~\ref{sec:impulse-controls} that a
\emph{control band} policy
$\varphi=\{d,D,U,u\}$ is optimal among all feasible policies.  Under
the control band policy $\varphi$, an adjustment is placed so as to
bring the inventory up to level $D$ when the inventory level drops to
level $d$ and to bring the inventory down to level $U$ when the
inventory level rises to level $u$.
For a singular
Brownian control problem, we show in Section~\ref{sec:instanteneous}
that the optimal policy is a \emph{degenerate} control band policy
with two free parameters $D=d$ and $U=u$.
When the inventory level is restricted to be always
nonnegative, we show in Section \ref{sec:nonnegative} that the optimal
policy for an impulse Brownian control problem is again a control band
policy. Depending on the holding cost function $h$, this control
band policy sometimes, but not always, has only three free parameters
$D$, $U$ and $u$ that need to be characterized with the lowest
boundary $d=0$. Although we will not explicitly study the mixed
impulse-singular Brownian control problems, it is clear from our
proofs that a degenerate control band policy with three parameters  is
optimal.

\subsection*{The Lower-Bound Approach and the Free Boundary Problem}

This paper promotes a three-step, lower bound approach to solving
Brownian control problems under the long-run average cost
criterion. In the first step, we prove Theorem \ref{thm:lowerbound}
showing that if there exist a constant $\gamma$ and a ``smooth'' test
function $f$ that is defined on the entire real line (or the positive
half line when inventory is not allowed to be backlogged) such that
$f$ and $\gamma$ jointly satisfy some differential inequalities, then
the long-run average cost under any feasible policy is at least
$\gamma$. This theorem is formulated and proved for all impulse,
singular and mixed impulse-singular control problems.  In the second
step, we show in Theorems \ref{thm:control band} and \ref{thm:controllimit} that for a given control band policy, its long-run
average cost can be computed as a solution to a Poisson equation. This
equation is a second order ordinary differential equation (ODE) with
given boundary conditions at the boundary points of the band. As a
part of the solution to the Poisson equation, we also obtain the
\emph{relative value function}. The relative value function can
naturally be extended to the entire real line, but the extended
function may not be continuously differential at the boundary points
of the control band. In the third step, we search for a control band
policy such that the corresponding relative value function can indeed be
extended smoothly as a function $f$ on the entire real line.
Furthermore, this smooth function $f$, together with the long-run average
cost under the control band policy, satisfies the differential
inequalities in step 1 within the entire real line. Clearly, if the
control band policy in step 3 can be found, it must be an optimal
policy by Theorem \ref{thm:lowerbound}. The lower-bound theorem,
Theorem~\ref{thm:lowerbound}, is known as the ``verification theorem''
in literature.

Step 3 is the most critical step in the three-step approach. In order
to make the relative value function smoothly extendible to the entire
real line, the parameters of the control band must be carefully
selected.  These parameters serve as the boundary points of the ODE and
they themselves need to be determined.  The smoothness requirements
impose conditions of the ODE solution at these yet to be found
boundary points. Thus, the ODE in step 3 is known as the free boundary
ODE problem. Solving the free boundary problem to find the optimal
parameters is also known as the ``smooth pasting'' method
\cite{BertolaCabellero90}.  Solving
a free boundary problem is often technically difficult.  The number of
free parameters of an optimal control band policy dictates the level
of difficulty in solving the free boundary problem.  Many papers in
the literature left it unsolved (e.g.~\cite{Dixit91,Richard77}),
assuming there
is a solution to the
free boundary problem with a certain smoothness property.

\subsection*{Contributions}

The Brownian inventory control problem is now a classical problem,
starting from Bather~\cite{Bather66} thirty five years ago. We will
survey the research area in the next several paragraphs. In addition
to providing a self contained tutorial on the lower-bound approach to
studying optimal control problems, our paper contributes significantly
in the following areas.  (a) Under a general convex holding cost
function with some minor assumptions, we rigorously prove the
existence of a control band policy that is optimal for both the
impulse and singular control problems under the long-run average cost
criterion.  (b) Under the general convex holding cost function, we
have proved the existence of a solution to the four-parameter
free-boundary problem.  Our existence proof leads naturally to
algorithms for computing optimal control band parameters.  These
algorithms reduce to root findings for continuous, monotone
functions. Thus, the convergence of these algorithms are guaranteed.
We are not aware of any paper that proved the existence of a solution
to the four-parameter free boundary problem under the long-run average
cost criterion.  In the discounted setting,
\cite{ConstantinidesRichard78} solved the four-parameter free boundary
problem when $h$ is linear, and \cite{Baccarin02} solved the problem
when $h$ is quadratic.  Recently, Feng and Muthuraman
\cite{FengMuthuraman10} developed an algorithm to numerically solve
the four-parameter free boundary problem for the discounted Brownian
control problem. They illustrate the convergence of their algorithm
through some numerical examples. However, the convergence of their
algorithm was not established.  (c) Under the long-run average cost
criterion, our lower-bound approach provides a unified treatment for
both the impulse and singular control problems, with and without
inventory backlog. In particular, we do not need to employ vanishing
discount approach \cite{schal93,FeinbergLewis07,WoongheeTimHuh11} to
study the long-run average cost problems.  In her book,
Stokey~\cite{Stokey09} summarizes both the impulse and instantaneous
controls of Brownian motion with a general convex holding cost
function. She focused on the discounted cost problems, and employed
the vanishing discount approach to deal with the long-run average cost
problems.  It is appealing that our current paper studies the long-run
average cost problem directly, and characterizes the optimal
parameters directly without going through the vanishing discount
procedure.


\subsection*{Literature Review}
\label{sec:literature-review}

The lower-bound approach were used in
\cite{Taksar85,OrmeciDaiVandeVate08} under a long-run average cost
criterion and in \cite{HarrisonSellkeTaksar83,HarrisonTaksar83} under
a discounted cost criterion.  The approach is essentially the same as
the quasi-variational inequality (QVI) approach that was pioneered by
Bensoussan and Lions~\cite{BensoussanLions75}. The QVI approach was
systematically developed in a French book that was later translated
into English (see \cite{BensoussanLions84}). An appealing feature of
the QVI approach is that it is sufficient to solve a QVI problem in
order to obtain an optimal policy for an inventory control problem,
and this sufficiency is established in \cite{BensoussanLions84}.  The
QVI problem is a pure analytical problem that is closely related the
free boundary problem.  Many authors directly start with the QVI
problems, relying on the ``verification theorem'' developed in
\cite{BensoussanLions84}; see for example,
\cite{Sulem86,BensoussanLiuSethi05,
  Benkherouf07, BenkheroufBensoussan09}.  The potential drawback of
this approach is that when the formulation of a Brownian control
problem is slightly different from the setting in
\cite{BensoussanLions84}, one may have to developed a new verification
theorem, presumably mimicking the development  in the
book. In contrast, our lower-bound approach allows us to provide a
self-contained, rigorous proof simultaneously for impulse, singular
and mixed control problems. It also allows one to directly see how the
smooth requirement  of a solution to the free-boundary problem is
used. We believe the lower-bound approach is easier to be generalized
to high dimensional Brownian control problems.

The impulse control problem with both upward and downward adjustments
was studied as early as 1976 and 1978 in two papers by
Constantinides~\cite{Constantinides76} and Constantinides et
al.~\cite{ConstantinidesRichard78}.  The first paper studies the
long-run average cost objective and the second paper studies the
discounted cost objective. Both papers assume the holding cost
function is linear as given in (\ref{eq:linearh}).  Under this
holding cost function, the optimal control band parameters can be
explicitly characterized.  Baccarin \cite{Baccarin02} studies
discounted impulse Brownian control problem with quadratic inventory
cost function.  When the inventory is restricted to be nonnegative,
but still under the linear holding cost assumption
(\ref{eq:linearh}), Harrison et al.~\cite{HarrisonSellkeTaksar83}
studies the discounted cost, impulse Brownian control problem whereas
Ormeci et al.~\cite{OrmeciDaiVandeVate08} studies the long-run average
cost problem. Under the linear holding cost function assumption, the
optimal policy is a degenerate control band policy $\{0, D, U, u\}$,
where three optimal parameters $D, U, u$ can be determined
explicitly. However, under our general convex holding cost
assumption, the optimal policy for the impulse control problem without
inventory backlog is again a control band policy $\{d, D, U, u\}$,
with $d$ sometimes being strictly positive.  Harrison and
Taksar~\cite{HarrisonTaksar83} and Taksar \cite{Taksar85} study the
singular Brownian control problem under a general convex inventory
cost function assumption. The former paper studies the discounted cost
problem and the latter studies the long-run average cost
problem. Taksar \cite{Taksar85} characterizes the optimal control band
parameters through the optimal stopping time to a stochastic game
without solving the two-parameter free boundary problem.
As in \cite{Taksar85}, Stokey \cite{Stokey09}
characterizes her optimal parameters through a stopping time
problem without solving  the four-parameter free boundary problem.
These stopping time characterizations do not easily lead to any
numerical algorithm to
compute two optimal parameters.  Richard~\cite{Richard77} studies an
impulse control of a general one-dimensional diffusion process. He
assumes without proof the existence of a solution to a
quasi-variational inequality problem with certain regularity property
in order to characterize an optimal policy.  Kumar and Muthuraman
\cite{KumarKumar04} develop a numerical algorithm to solve
high-dimensional singular control problems.  Vickson \cite{Vickson86}
studies a cycling problem with Brownian motion demand.


In his pioneering paper, Bather \cite{Bather66} studies the impulse
Brownian motion control problem \emph{without} downward adjustment,
under the long-run average cost criterion. For most inventory
problems, without downward adjustment is a natural setting.  Under a
general holding cost function, he suggests that an ($s,S$) policy is
optimal and derives equations that characterize  the optimal parameters
$s$ and $S$.  Many authors have generalized this paper to various
settings, to discounted cost problems with linear holding cost in
\cite{Sulem86}, to discounted cost problems with and without inventory
backlog in \cite{Chao92}, to discounted cost problems under the
general convex holding cost function assumption in
\cite{Benkherouf07}, to discounted cost problems with positive
constant leadtime in \cite{Bar-IlanSulem95}, to compound Poisson and
diffusion demand processes in
\cite{BensoussanLiuSethi05,BenkheroufBensoussan09}.  Because
there is no downward adjustment in these problems, the optimal policy
has two parameters and the resulting two-parameter free boundary
problem can be solved much easier than the four-parameter one.

\subsection*{Paper Organization}

The rest of this paper is organized as follows.  In Section
\ref{sec:model}, we define our Brownian control problem in a unified
setting that includes impulse, singular and mixed impulse-singular
controls.  In Section \ref{sec:ito-formula-lower} we present a version
of It\^{o} formula that does not require the test function  $f$ be
$C^2$ function.  A lower
bound for all feasible policies is established in Section
\ref{sec:lowerbound}.  Section~\ref{sec:impulse-controls} devotes to
impulse  control problems that allow inventory backlog under the
long-run average cost criterion. Section \ref{sec:controlBand} shows
that under a control band policy, a Poisson equation can produce a
solution that gives both the long-run average cost and the
corresponding relative value function. Under the assumption that a
free-boundary problem has a unique solution that has desired
regularity properties, Section \ref{sec:optimal} proves that there is
a control band policy whose long-run average cost achieves  the lower
bound. Thus, the control band policy is optimal among all feasible
policies.  Section~\ref{sec:optimal-control-band} is a lengthy one
that devotes to the existence proof of the solution to the free-boundary
problem. In the section, the parameters for the optimal control band
policy are characterized.  Section~\ref{sec:optimal-control-band}
constitutes the main technical contribution of this paper.
Section~\ref{sec:instanteneous} solves the singular control
problem. This section is short, essentially becoming a special case of
Section \ref{sec:impulse-controls} when both $K=0$ and $L=0$.
Section~\ref{sec:nonnegative} deals with impulse control problems when
inventory is not allowed backlogged. 
  Finally,
Section \ref{sec:conclusions} summarizes  this paper and discusses a
few extensions.

\section{Brownian Control Models}
\label{sec:model}

Let $X=\{X(t), t\ge 0\}$ be a Brownian motion
with drift $\mu$ and variance $\sigma^2$, starting from
$x$. Then, $X$ has the following representation
\begin{displaymath}
  X(t) =  x + \mu t + \sigma W(t), \quad t\ge 0,
\end{displaymath}
where $W=\{W(t), t\ge 0\}$ is a standard Brownian
motion that has drift
$0$, variance $1$, starting from $0$.
We assume $W$ is  defined on  some filtered probability space
$(\Omega, \{{\cal F}_t\},  {\cal F}, \Prob)$ and $W$ is an
 $\{{\cal F}_t\}$-martingale. Thus, $W$ is also known as  an  $\{{\cal
   F}_t\}$-standard Brownian motion.
 We use $X$ to model the
\emph{netput process} of the firm.  For each $t\ge 0$, $X(t)$
represents the inventory level at time $t$ if no control has been
exercised by time $t$. The netput process will be controlled and the
actual inventory level at time $t$, after controls has been exercised,
is denoted by $Z(t)$.  The controlled process is denoted by $Z=\{Z(t),
t\ge 0\}$. With a slight abuse of terminology,
we call $Z(t)$ the inventory level at time $t$, although when
$Z(t)<0$, $\abs{Z(t)}$ is the backorder level at time $t$.

Controls are dictated by a policy.  A policy $\varphi$ is a pair of
stochastic processes $(Y_1, Y_2)$ that satisfies the following three
properties: (a) for each sample path $\omega\in \Omega$, $Y_i(\omega,
\cdot)\in \D$, where $\D$ is the set of  functions on
$\R_+=[0,\infty)$ that are right continuous on $[0, \infty)$
and have left limits in $(0, \infty)$, (b) for each $\omega$,
$Y_i(\omega, \cdot)$ is a nondecreasing function, (c) $Y_i$ is adapted
to the filtration $\{{\cal F}_t\}$, namely, $Y_i(t)$ is ${\cal
  F}_t$-measurable for each $t\ge 0$.  We call $Y_1(t)$ and $Y_2(t)$
the cumulative \emph{upward} and \emph{downward} adjustment,
respectively, of the
inventory in $[0, t]$. Under a given policy $(Y_1, Y_2)$, the
inventory level at time $t$ is given by
\begin{equation}
  \label{eq:semimartingaleRep}
  Z(t) = X(t) +Y_1(t) -Y_2(t)=x+\sigma W(t)+\mu t +Y_1(t)-Y_2(t) ,
  \quad t\ge 0.
\end{equation}
Therefore,  $Z$ is a semimartingale, namely,
a martingale $\sigma W$ plus a process that is of bounded variation.

A point $t\ge 0$ is said to be an \emph{increasing point} of $Y_1$ if
$Y_1(s)-Y_1(t-)>0$ for each $s>t$, where $Y_1(t-)$ is the left limit
of $Y_1$ at $t$ with convention that $Y_1(0-)=0$. When $t$ is an
increasing point of $Y_1$, we call it an upward adjustment time.
Similarly, we define an increasing point of $Y_2$ and call it a
downward adjustment time.  Let $N_i(t)$ be the cardinality of the
set
\begin{displaymath}
\{s\in [0, t]: Y_i \text{ increases at $s$}\}, \quad i=1, 2.
\end{displaymath}
In general, we allow an upward or downward adjustment at time
$t=0$.  By convention,  we set $Z(0-)=x$ and call $Z(0-)$ the
\emph{initial  inventory level}. By (\ref{eq:semimartingaleRep}),
\begin{displaymath}
  Z(0) = x + Y_1(0)-Y_2(0),
\end{displaymath}
which can be different from the initial inventory level $Z(0-)$.

There are two types of costs associated with a control. They are
\emph{fixed costs} and \emph{proportional costs}. We assume that each
upward adjustment incurs a fixed cost of $K\ge 0$ and each downward
adjustment incurs a fixed cost of $L\ge 0$.  In addition, each unit of
upward adjustment incurs a proportional cost of $k>0$  and each unit
of downward adjustment incurs a proportional cost of $\ell>0$. Thus, by
time $t$, the system incurs the cumulative proportional cost $k
Y_1(t)$ for upward adjustment and the cumulative proportional cost
$\ell Y_2(t)$ for downward adjustment. When $K>0$, we are only
interested in policies such that $N_1(t)<\infty$ for each $t>0$;
otherwise, the total cost would be infinite in the time
interval $[0, t]$. Thus, when $K>0$, we restrict upward controls that
have a
finitely many upward adjustment in a finite interval.
This is equivalent to requiring $Y_1$ to be a piecewise constant
function on each sample path.  Under such an
upward control, the upward adjustment times can be listed as a
discrete sequence $\{T_1(n):n\ge 0\}$,  where the $n$th upward
adjustment time can be defined recursively via
\begin{displaymath}
  T_1(n) = \inf\{t> T_1(n-1): \Delta Y_1(t)>0\},
\end{displaymath}
where, by convention, $T_1(0)=0$ and $\Delta Y_1(t)=Y_1(t)-Y_1(t-)$.
The amount of the $n$th upward adjustment is denoted by
\begin{displaymath}
  \xi_1(n)= Y_1(T_1(n))-Y_1(T_1(n)-) \quad n=0, 1, \ldots.
\end{displaymath}
It is clear that specifying such a upward adjustment policy $Y_1=\{Y_1(t), t\ge
0\}$ is equivalent to specifying a sequence of $\{(T_1(n), \xi_1(n)):
n\ge 0\}$. In particular, given the sequence, one has
\begin{equation}
  \label{eq:Y1N1xi1}
  Y_1(t) = \sum_{i=0}^{N_1(t)} \xi_1(i),
\end{equation}
and $N_1(t)=\max\{n\ge 0:T_1(n)\le t\}$. Thus, when $K>0$, it is sufficient to specify the sequence
$\{(T_1(n), \xi_1(n)):
n\ge 0\}$  to describe an upward adjustment policy. Similarly, when
$L>0$, it is sufficient to specify the sequence
$\{(T_2(n), \xi_2(n)):
n\ge 0\}$  to describe a downward adjustment policy
and
\begin{equation}
  \label{eq:Y2N2xi2}
  Y_2(t) = \sum_{i=0}^{N_2(t)} \xi_2(i).
\end{equation}
 Merging these two
sequences, we have  the
sequence $\{(T_n, \xi_n), n\ge 0\}$, where
$T_n$ is the $n$th adjustment time of the inventory and $\xi_n$ is the
amount of adjustment at time $T_n$. When $\xi_n>0$, the $n$th
adjustment is an upward adjustment and when $\xi_n<0$, the $n$th
adjustment is a downward adjustment. The policy  $(Y_1, Y_2)$ is
adapted if  $T_n$ is an $\{{\cal F}_{t}\}$-stopping time and
 each adjustment $\xi_n$ is $\mathscr{F}_{T_n-}$ measurable,

In addition to the adjustment cost, the system is assumed to
incur  the holding cost at rate $h(x)$: when the inventory level
is at $Z(t)=x$, the system incurs a cost of $h(x)$ per unit of time.
Therefore, the cumulative holding cost in $[0, t]$ is
\begin{displaymath}
  \int_0^t h(Z(s)) ds.
\end{displaymath}
Under a feasible policy $\varphi=\{(Y_1(t), Y_2(t)\}$
with initial inventory level $Z(0-)=x$,
the long-run average cost $\AC(x, \varphi)$ is
\begin{equation}
\AC(x,\varphi)=\limsup_{t\rightarrow
  \infty} \frac{1}{t}\mathbb{E}_x\Big[\int_0^{t}h(Z(s))ds+
KN_1(t)+LN_2(t) + k Y_1(t)+\ell Y_2(t)
\Big],\label{eq:AC-1}
\end{equation}
where $\E_x$ is the expectation operator conditioning the initial
inventory level $Z(0-)=x$.
As mentioned earlier,  when $K>0$ and $L>0$, it is sufficient to
restrict feasible policies to be impulse type given
 in (\ref{eq:Y1N1xi1}) and (\ref{eq:Y2N2xi2}). Such a Brownian
inventory control model is called the \emph{impulse Brownian control model}.
When $K=0$ and $L=0$, it turns out the under an optimal policy,
 $N_1(t)=\infty$ and $N_2(t)=\infty$ with positive probability for
 each $t>0$. The corresponding control problem is called the
 \emph{instantaneous Brownian control model} or \emph{singular Brownian
   control model}.

In this paper, we make the following assumption on the holding cost
function $h:\R\to \R_+$.
\begin{assumption} \label{assumption:h}
Assume that the continuous holding cost function $h:\R\to \R^+$ satisfies the
following conditions:
(a) it is convex; (b) there
exists an $a$ such
that $h\in C^2(\R)$ except at $a$, and $h(a)=0$; (c)  $h'(x)< 0$ for $x<a$ and
$h'(x)> 0$ for $x>a$; 
(d) When $\lambda=\frac{2\mu}{\sigma^2}\neq 0$,  $h'(x)$ has
smaller order than $e^{-\lambda x}$, that is
\begin{equation}
  \label{eq:hIntegral}
  \int_{-\infty}^a \abs{h'(y)} e^{\lambda(y-a)}dy <\infty \text{ if }
  \lambda=\frac{2\mu}{\sigma^2}>0
\end{equation}
and
\begin{equation}
  \label{eq:hIntegral2}
  \int_a^{\infty} \abs{h'(y)} e^{\lambda(y-a)}dy <\infty \text{ if }
  \lambda=\frac{2\mu}{\sigma^2}<0.
\end{equation}
\end{assumption}

We only consider  feasible  policies that satisfy
\begin{eqnarray}
  \label{eq:regularPolicy1}
&& \mathbb{E}_x\bigl[ Y_i(t)\bigr]<\infty \quad  i=1,2,\\
&& \E_x[N_1(t)]<\infty \text{ when }K>0 \quad \text{and} \quad
\E_x[N_2(t)]<\infty \text{ when } L>0
  \label{eq:regularPolicy2}
\end{eqnarray}
for each $t\ge 0$.
Otherwise, $\AC(x, \varphi)=\infty$.
In some applications,
one might require inventory level be nonnegative always, namely,
\begin{eqnarray*}
Z(t)\geq 0,\ \ \mbox{for } t\geq 0.
\end{eqnarray*}

\section{The It\^{o} Formula}
\label{sec:ito-formula-lower}
In this section, we first state a version of It\^{o}'s formula. We
then provide a lower bound result for the long-run average cost in
(\ref{eq:AC-1}).  Recall that for a function $g\in\D$,
it is right
continuous on $[0, \infty)$ and has left limits in $(0,\infty)$.
 We use $g^c$ to denote the continuous part of $g$, namely,
\begin{displaymath}
  g^c(t)= g(t) - \sum_{0\le s \le t} \Delta g(s) \quad \text{ for }
  t\ge 0.
\end{displaymath}
Here we assume $g(0-)$ is well defined.
Recall under any feasible policy $\varphi=(Y_1, Y_2)$, the inventory
process $Z=\{Z(t):t\ge 0\}$ has the semimartingale representation
(\ref{eq:semimartingaleRep}). Because Brownian motion has continuous
sample paths, we have
\begin{equation}
  \label{eq:Zc}
  Z^c(t) = X(t) + Y_1^c(t) -Y_2^c(t)\quad \text{ for }t\ge 0.
\end{equation}


\begin{lemma}
\label{lem Ito}
Assume that $f\in C^1(\R)$ and $f'$ is absolutely continuous such that
$f'(b)-f'(a)=\int_a^bf''(u)du$ for any $a<b$ with $f''$ locally in $L^1$.
Then
\begin{eqnarray}
  f(Z(t)) & =& 
f(Z(0))+\int_0^t \Gamma f(Z(s))ds + \sigma \int_0^t
  f'(Z(s))dW(s) \nonumber\\
&&{}+ \int_0^t f'(Z(s-)) dY_1^c(s) - \int_0^t
  f'(Z(s-))dY^c_2(s)+ \sum_{0< s \le t} \Delta
f(Z(s)), \label{eq:ito}
\end{eqnarray}
where
\begin{equation}
  \label{eq:Gamma}
\Gamma f(x) = \frac{1}{2}\sigma^2 f''(x) + \mu f'(x), \quad  \text{
  for each } x\in \R \text{ such that $f''(x)$ exists},
\end{equation}
is the generator of the $(\mu, \sigma^2)$-Brownian motion $X$
and $\int_0^t f'(Z(s))dW(s)$  is interpreted as the It\^{o} integral.
\end{lemma}
\begin{remark}
  Although $f''(u)$ is only defined on almost all $u$ in
  $\R$, $\int_0^t f''(Z(s))ds$ is uniquely defined almost
  surely. Indeed,
  \begin{displaymath}
    \sigma^2\int_0^t f''(Z(s))ds = \frac{1}{2} \int_{\R} f''(a) L^a(t)da,
  \end{displaymath}
where $L^a$ is the local time of $Z$ at $a$.
\end{remark}
\begin{proof}
 For any semimartingale $Z$, it follows from Theorem 71 of
 \cite[pp.~221]{Protter05} and the comment of \cite[pp.~70]{Protter05}
 that
\begin{eqnarray}
  f(Z(t)) & = & f(Z(0)) + \int_0^t f'(Z(s-))dZ^c(s) + \frac{1}{2}
  \int_0^t
  f''(Z(s-)) d[Z^c, Z^c](s) \nonumber\\
&& {} + \sum_{0< s\le t} \Delta f(Z(s)),\label{eq:orginalIto}
\end{eqnarray}
where $[Z^c, Z^c]$ is the quadratic variation of $Z^c$. Using
semimartingale representation
(\ref{eq:Zc}), we have
\begin{equation}
  \label{eq:ZcZc}
[Z^c, Z^c](t)=[X, X](t)=\sigma^2 t
\end{equation}
  and
\begin{eqnarray}
  \int_0^t f'(Z(s-))dZ^c(s) & = &
  \int_0^t f'(Z(s-))\mu ds +   \int_0^t f'(Z(s-))\sigma dW(s)
  \nonumber \\ && {}+
  \int_0^t f'(Z(s-))d Y_1^c(s) -   \int_0^t f'(Z(s-))d
  Y_2^c(s) \label{eq:fZmZ}
\end{eqnarray}
 for $t\ge 0$. Because $Y_1$ and $Y_2$ have at most countably many
 jump points, $Z$ has at most countably many discontinuity points.
Therefore, we have
\begin{equation}
  \label{eq:fZmds}
      \int_0^t f'(Z(s-)) ds  =   \int_0^t f'(Z(s))ds
\quad \mbox{and} \quad
  \int_0^t f'(Z(s-)) dW(s) =   \int_0^t f'(Z(s)) dW(s)
\end{equation}
for all $t\ge 0$ almost surely. It\^{o} formula (\ref{eq:ito})
then follows from (\ref{eq:orginalIto})-(\ref{eq:fZmds}).

\end{proof}

\section{Lower Bound}
\label{sec:lowerbound}
In this section, we state and prove a theorem that  establishes a
lower bound for the optimal long-run
average cost. This theorem is closely related to the ``verification
theorem'' in literature. Its proof is self contained, using the It\^o
lemma in Section~\ref{sec:ito-formula-lower}.

\begin{theorem}\label{thm:lowerbound}
  Suppose that $f\in C^1(\R)$ and $f'$ is absolutely continuous such
  that $f''$ is
  locally $L^1$. Suppose that there exists a constant $M>0$ such
  that  $|f'(x)|\leq M $ for all $x\in \R$.
 Assume further that
  \begin{eqnarray}
    && \Gamma f(x)+h(x)\geq \gamma \mbox{ for almost all $x\in
      \mathbb{R}$},\label{eq:lbPoission}\\
    && f(y) - f(x)\le K+ k(x-y) \mbox{ for $y<x$},\label{eq:lbK} \\
    && f(y) - f(x)\le L+\ell(y-x) \mbox{ for $x<y$}.\label{eq:lbL}
  \end{eqnarray}
Then $\AC(x,\varphi)\geq \gamma$ for each feasible policy $\varphi$  and each
initial state $x\in \mathbb{R}$.
\end{theorem}
\begin{remark}
(i) When $K=0$, condition (\ref{eq:lbK}) is equivalent to  that $f'(x) \ge
-k$ for each $x\in \R$. When $L=0$, condition (\ref{eq:lbL}) is
equivalent to that $f'(x)\le \ell$ for each $x\in\R$.
(ii)
Because under an arbitrary control policy, the
inventory level $Z$ can potentially reach any level. Thus, we require
function $f$ to be defined on the entire real line $\R$. It is
\emph{not} enough to have $f$ defined on a certain interval $[d, u]$.
\end{remark}
\begin{proof}
Let $\varphi=(Y_1, Y_2)$ be a feasible
policy.  We choose a version of $f''(x)$ such that (\ref{eq:lbPoission})
holds for every $x\in\R$.
By It\^{o}'s
formula (\ref{eq:ito}),
\begin{eqnarray}
  f(Z(t)) &=&  f(Z(0-))+ \int_0^t \Gamma f(Z(s))ds + \sigma \int_0^t
  f'(Z(s))d W(s) + \int_0^t f'(Z(s-))dY_1^c(s) \nonumber \\ && {}- \int_0^t
  f'(Z(s-))dY_2^c(s) + \sum_{0\le s \le t} \Delta f(Z(s)) \nonumber\\
 &\ge &  f(Z(0-))+ \gamma t- \int_0^t h (Z(s))ds + \sigma \int_0^t
  f'(Z(s))d W(s) + \int_0^t f'(Z(s-))dY_1^c(s) \nonumber \\ && {}- \int_0^t
  f'(Z(s-))dY_2^c(s) + \sum_{0\le s \le t} \Delta f(Z(s)) \label{eq:itoinequality}
\end{eqnarray}
where the inequality is due to (\ref{eq:lbPoission}). In the rest of
the proof, we separate into different cases depending on the positivity
of $K$ and $L$. We will provide a complete proof for the case when
$K>0$ and $L>0$. Sketches will be provided for proofs in  other cases.

\textbf{Case I: Assume that $K>0$ and $L>0$}.  In this case,
it is sufficient to restrict feasible policies  to impulse
control policies $\{(T_n, \xi_n): n=0, 1, \ldots \}$. In this case, $Y_1^c=0$
and $Y_2^c=0$. Conditions (\ref{eq:lbK}) and (\ref{eq:lbL}) imply
that
 and $\Delta f (Z(T(n)))\ge - \phi(\xi_n)$  for $n=0, 1, \ldots$, where
\begin{displaymath}
\phi(\xi)= \left \{
\begin{array}{ll}
   K+k\xi, & \mbox{if }\xi>0, \\
   0, & \mbox{if }\xi=0, \\
   L-l\xi, & \mbox{if }\xi<0.
  \end{array}
\right.
\end{displaymath}
Therefore, (\ref{eq:itoinequality}) leads to
\begin{equation}
  \label{eq:itoinequality2}
  f(Z(t)) \ge f(Z(0-)) + \gamma t -\int_0^t h(Z(s))ds + \sigma \int_0^t
  f'(Z(s)) dW(s) - \sum_{n=0}^{N(t)} \phi(\xi_n)
\end{equation}
for each $t\ge 0$.
Fix an $x\in \R$. We assume that
\begin{displaymath}
  \E_x\biggl ( \int_0^t h (Z(s))ds + \sum_{n=0}^{N(t)}\phi(\xi_n) \biggr )< \infty
\end{displaymath}
for each $t>0$. Otherwise, $\AC(x, \varphi )=\infty$ and thus $\AC(x,
\varphi)\ge \gamma$  is trivially satisfied. Because $\abs{f'(x)}\le
M$, $\E_x\abs{\int_0^t f'(Z(s))dW(t)}<\infty$ and
$\E_x{\int_0^tf'(Z(s))dW(s)}=0$. Meanwhile
\begin{displaymath}
f(Z(t)) \le \bigl
(f(Z(t)) \bigr )^+
\end{displaymath}
and $\E_x \bigl[\bigl(f(Z(t)) \bigr )^+\bigl]$ is well defined, though
it can be $\infty$,  where, for a $b\in \R$, $b^+=\max(b, 0)$.
Taking $\E_x$ on the both
 sides of (\ref{eq:itoinequality2}), we have
 \begin{displaymath}
  \E_x\bigl[\bigl(f(Z(t)) \bigr )^+\bigl]\ge \E_x\bigl[f(Z(0-)) \bigl] + \gamma t -
  \E_x\biggl ( \int_0^t h (Z(s))ds + \sum_{n=0}^{N(t)}\phi(\xi_n) \biggr ).
 \end{displaymath}
Dividing both sides by $t$ and taking limit as $t\to\infty$, one has
\begin{equation}
  \label{eq:interInequality}
\liminf_{t\to\infty} \frac{1}{t}  \left[ \E_x\biggl ( \int_0^th (Z(s))ds +
\sum_{n=0}^{N(t)}\phi(\xi_n) \biggr )  +  \E_x\bigl[\bigl(f(Z(t)) \bigr )^+\bigl] \right] \ge \gamma.
\end{equation}
We consider two cases. In the first case when
\begin{displaymath}
  \liminf_{t\to\infty} \frac{1}{t} \E_x\bigl[\bigl(f(Z(t)) \bigr )^+\bigl]= 0,
\end{displaymath}
it is clear that (\ref{eq:interInequality}) implies the theorem.
Now we consider the case when
\begin{displaymath}
  \liminf_{t\to\infty} \frac{1}{t} \E_x\bigl[\bigl(f(Z(t)) \bigr )^+\bigl]=b> 0.
\end{displaymath}
It follows that for sufficiently large $t$,
\begin{eqnarray}
\label{eq:E(f(Z))>}
   \E_x\bigl[\bigl(f(Z(t)) \bigr )^+\bigl]\ge (b/2) t.
\end{eqnarray}
Because $\abs{f'(y)}\le M$, for all $y\in \R$,
\begin{displaymath}
(f(y_1))^+ -(f(y_2))^+ \le  \abs{f(y_1)-f(y_2)}\le M \abs{y_1-y_2} \le M
(\abs{y_1}+\abs{y_2}).
\end{displaymath}
Therefore,
\begin{displaymath}
  \abs{Z(t)} \ge \frac{1}{M}\bigr(f(Z(t))^+ - (f(Z(0)))^+\bigr)-|Z(0)|,
\end{displaymath}
which, together with (\ref{eq:E(f(Z))>}), implies that
\begin{eqnarray*}
  \E_x\abs{Z(t)} &\ge&
  \frac{1}{M}\bigr(\E_x[(f(Z(t)))^+] - \E_x[(f(Z(0)))^+]\bigr)-\E_x\abs{Z(0)}\\
&\ge& \frac{1}{M}\bigl((b/2)t - \E_x[(f(Z(0)))^+]\bigr)-\E_x\abs{Z(0)},
\end{eqnarray*}
for sufficiently large $t$. This implies that
\begin{equation}
  \label{eq:Zinfinite}
  \liminf_{t\to\infty} \frac{1}{t} \int _0^t \E_x\abs{Z(s)} ds
  =\infty.
\end{equation}
Now we prove that
\begin{equation}
  \label{eq:inventoryInfinite}
    \liminf_{t\to\infty} \frac{1}{t} \int _0^t \E_x\bigl[h(Z(s))\bigr] ds =\infty,
\end{equation}
which implies that $\AC(x, \varphi)=\infty$, thus proving the theorem.

To see \eqref{eq:inventoryInfinite}, by the Assumption (a) and (c), there
exist  constants $h_1>0$ and $c>0$ such
that
\begin{equation}
  \label{eq:hDerivatie}
  h'(y) \ge h_1 \text{ for all }  y \ge c \quad \text{and} \quad
  h'(y) \le -h_1 \text{ for all }  y \le -c.
\end{equation}
Because  of \eqref{eq:Zinfinite}, one of the following two equations
holds:
\begin{equation}
  \label{eq:Zinfiniteplus}
  \liminf_{t\to\infty} \frac{1}{t} \E_x  \biggl (\int _0^t
  Z(s)1_{\{Z(s)\ge c\}} ds  \biggr )
  =\infty,
\end{equation}
\begin{equation}
  \label{eq:Zinfiniteminus}
  \liminf_{t\to\infty} \frac{1}{t} \E_x  \biggl ( \int _0^t
  |Z(s)|1_{\{Z(s)\le -c\}} ds  \biggr )
  =\infty.
\end{equation}
Assume that \eqref{eq:Zinfiniteminus} holds. Condition
\eqref{eq:hDerivatie} implies that
\begin{displaymath}
 h(-c)-h(y) \le (-h_1)(-c-y)   \quad \text{ for } y \le -c
\end{displaymath}
or
\begin{displaymath}
  h(y) \ge h_1 \abs{y} +h(-c) - ch_1 \text{ for } y \le -c.
\end{displaymath}
Therefore,
\begin{displaymath}
  h(y)1_{\{y\le -c\}} \ge h_1|y|1_{\{ y\le -c\}} -ch_1.
\end{displaymath}
It follows that
\begin{eqnarray*}
\liminf_{t\to\infty} \frac{1}{t} \E_x \biggl ( \int _0^t h(Z(s)) ds
\biggr ) &\ge &
\liminf_{t\to\infty} \frac{1}{t} \E_x  \biggl (\int _0^t h(Z(s))1_{\{Z(s)\le
  -c\}} dt  \biggr ) \\
&\ge & \liminf_{t\to\infty} \frac{1}{t}  \biggl ( \E_x\int _0^t
h_1|Z(s)|1_{\{Z(s)\le -c\}} ds  \biggr ) \\
&  = &\infty,
\end{eqnarray*}
which proves \eqref{eq:inventoryInfinite}.
Hence the theorem is proved for $K>0$ and $L>0$.

\textbf{Case II: Assume that $K=0$ and $L=0$.}
 Condition (\ref{eq:lbK}) leads to
$f'(u)\ge -k$ for all $u\in \R$ and condition (\ref{eq:lbL}) leads to
$f'(u)\le \ell$ for all $u\in \R$.
Because $f$ is continuous, $\Delta f(Z(s))\neq 0$ implies that $\Delta
Z(s)\neq 0$. If $\Delta Z(s)>0$, (\ref{eq:lbK}) implies that
\begin{displaymath}
  \Delta f(Z(s))  \ge -k \Delta Z(s).
\end{displaymath}
If $\Delta Z(s)<0$, (\ref{eq:lbL}) implies that
\begin{displaymath}
  \Delta f(Z(s)) \ge -\ell \Delta Z(s).
\end{displaymath}
Thus, the last three terms
in (\ref{eq:itoinequality}) is at least
\begin{eqnarray*}
\lefteqn{  -k Y_1^c(t) - \ell Y_2^c(t) + \sum_{0\le s \le t} \Delta
  f(Z(s))} \\
&& {} \ge -k Y_1^c(t) - \ell Y_2^c(t) -k \sum_{0\le s \le t\atop
  \Delta Z(s)>0} \Delta Z(s)
 -\ell \sum_{0\le s \le t\atop
  \Delta Z(s)<0} \Delta Z(s) \\
&& {}  = -k Y_1^c(t) - \ell Y_2^c(t) -k \sum_{0\le s \le t} \Delta Y_1(s)
 -\ell \sum_{0\le s \le t} \Delta Y_2(s) \\
&& {}\ge -k Y_1(t) - \ell Y_2(t).
\end{eqnarray*}
Therefore, (\ref{eq:itoinequality}) leads to
\begin{equation*}
  \label{eq:itoinequality3}
  f(Z(t)) \ge f(Z(0-)) + \gamma t - \int_0^th(Z(s))ds + \sigma \int_0^t
  f'(Z(s)) d W(s) -k Y_1(t) -\ell Y_2(t)
\end{equation*}
for $t\ge 0$.
The rest of the proof  is identical to the case when
$K>0$ and $L>0$.

\textbf{Case III: Assume $K>0$ and $L=0$.} Consider a feasible policy
$(Y_1, Y_2)$ with a finite cost. The upward
controls must be impulse controls and $Y_1(t)=\sum_{n=0}^{N_1(t)}
\xi_1(n)$. Condition (\ref{eq:lbK})
implies that
\begin{displaymath}
  \sum_{0\le s \le t \atop \Delta Z(s)>0} \Delta f(Z(s)) \ge
  -\sum_{n=0}^{N_1(t)} (K+ k \xi_1(n)).
\end{displaymath}
and condition (\ref{eq:lbL}) implies that
\begin{eqnarray*}
 - \ell Y_2^c(t) + \sum_{0\le s \le t \atop  \Delta Z(s)<0} \Delta f(Z(s))
\ge  - \ell Y_2(t).
\end{eqnarray*}
Therefore, (\ref{eq:itoinequality}) leads to
\begin{equation*}
  \label{eq:itoinequality4}
  f(Z(t)) \ge f(Z(0-)) + \gamma t - \int_0^th(Z(s))ds + \sigma \int_0^t
  f'(Z(s)) d W(s)   -\sum_{n=0}^{N_1(t)} (K+ k \xi_1(n))  -\ell Y_2(t)
\end{equation*}
for $t\ge 0$.
The rest of the proof  is identical to the case when
$K>0$ and $L>0$.

\textbf{Case IV: Assume that $K=0$ and $L>0$}. This case is analogous
to the case when $K>0$ and $L=0$. Thus, the proof is omitted.
\end{proof}

\section{Impulse Controls}
\label{sec:impulse-controls}
In this section, we assume that  $K>0$ and $L>0$. Therefore, we
restrict our feasible policies to impulse controls as in
(\ref{eq:Y1N1xi1}) and (\ref{eq:Y2N2xi2}).
An impulse control band policy is defined by four
parameters $d$, $D$, $U$, $u$, where $d<D< U<u$. Under the policy,
when the inventory
falls to $d$, the system instantaneously orders items to bring it to
level $D$; when the inventory rises to $u$, the system adjusts its
inventory to bring it down to $U$. Given a control band policy
$\varphi$, in Section~\ref{sec:controlBand} we provide a method for
performance evaluation. As a byproduct, we also obtain the relative
value function associated with the control band policy.
Then in Section~\ref{sec:optimal} we show that an optimal policy is a
control band policy and present equations that uniquely determine
the optimal control band parameters $(d^*, D^*, U^*, u^*)$.

\subsection{Control Band Policies}
\label{sec:controlBand}
We use  $\{d, D, U, u\}$ to denote the
control band policy associated with parameters $d$, $D$, $U$, $u$.
Let us fix a control band policy $\varphi=\{d, D, U, u\}$
and an initial inventory level $Z(0-)=x$.
The adjustment amount $\xi_n$ of
the control band policy is given by
\begin{displaymath}
 \xi_0=
   \begin{cases}
   D-x, & \mbox{if } x\leq d, \\
   0, & \mbox{if } d<x<u, \\
   U-x, & \mbox{if }x\geq u,
   \end{cases}
\end{displaymath}
and for $n=1,2,...$,
\begin{displaymath}
  \xi_n=
  \begin{cases}
   D-d, & \mbox{if }Z(T_n-)=d, \\
   U-u, & \mbox{if }Z(T_n-)=u,
  \end{cases}
\end{displaymath}
where again $Z({t-})$ denotes the left limit at time $t$, $T_0=0$ and
\begin{displaymath}
  T_n = \inf\bigl\{ t> T_{n-1}: Z(t)\in\{d, u\}\bigr\}
\end{displaymath}
is the $n$th adjustment time.
 (By convention, we
assume $Z$ is right continuous having left limits.)
Our first task is to find its long-run average cost. We first present the following theorem.

\begin{theorem}
\label{thm:control band}
Assume that a control band policy $\varphi=\{d, D, U, u\}$ is fixed. If
there
exist a constant $\gamma$ and  a twice continuously differentiable
function
$V:[d,u]\rightarrow \R$ that satisfies
\begin{equation}
\Gamma V(x)+h(x)=\gamma, \quad d\leq x\leq u,\label{Poisson equ}
\end{equation}
with boundary conditions
\begin{eqnarray}
 && V(d)-V(D)=K+k(D-d),\label{V(d)-V(D)}\\
 && V(u)-V(U)=L+l(u-U),\label{V(u)-V(U)}
\end{eqnarray}
then the average cost $AC(x,\varphi)$ is independent of the starting
point $x\in \R$ and is given by $\gamma$ in \eqref{Poisson equ}.
\end{theorem}
\noindent
\begin{remark}
Equation (\ref{Poisson equ}) is known as the Poisson equation.
The solution $V$ is known as a {\em relative value function} associated with the control band policy $\varphi$.
It is unique
up to a constant. One can evaluate $\gamma$ from \eqref{Poisson equ}
by taking $x$ to be any value in $[d, u]$.
\end{remark}
\begin{proof}
  Consider the control band policy $\varphi=\{d, D, U, u\}$. Let $V$
  be a twice continuously differentiable function on $[d, u]$ that
  satisfies (\ref{Poisson equ})-(\ref{V(u)-V(U)}).
Because $d\le Z(t)\le u$, by Lemma~\ref{lem    Ito},  we have
\begin{displaymath}
\mathbb{E}_x[V(Z(t))]=\mathbb{E}_x[V(Z(0))]+\mathbb{E}_x\Big[\int_0^t \Gamma
V(Z(s))ds\Big] +\mathbb{E}_x\Big[\sum_{n=1}^{N(t)}\theta_n \Big],
\end{displaymath}
where $\theta_n=V(Z(T_n))-V(Z(T_n-))$. Boundary conditions
(\ref{V(d)-V(D)}) and
(\ref{V(u)-V(U)}) imply that
$\theta_n=V(Z(T_n))-V(Z(T_n-))=-\phi(\xi_n)$ for $n=1, 2, \ldots$. Therefore,
\begin{eqnarray*}
\mathbb{E}_x[V(Z(t))]-\mathbb{E}_x[V(Z(0))]&=&\mathbb{E}_x\Big[\int_0^t \Gamma
V(Z(s))ds\Big] +\mathbb{E}_x\Big[\sum_{n=1}^{N(t)}\theta_n
\Big]\nonumber\\&=&\gamma t-\mathbb{E}_x\Big[\int_0^t h(Z(s))ds\Big]
-\mathbb{E}_x\Big[\sum_{n=1}^{N(t)}\phi(\xi_n) \Big]. \nonumber
\end{eqnarray*}
Dividing both sides by $t$ and letting $t\rightarrow\infty$, we have
$AC(x,\varphi)=\gamma$ because
\begin{displaymath}
\lim_{t\rightarrow\infty}\frac{1}{t}\mathbb{E}_x[V(Z(t))]=0\quad
\text{and}\quad
\mathbb{E}_x[V(Z(0))]=V(x+\xi_0).
\end{displaymath}
\end{proof}

We end this section by explicitly finding a solution $(V, \gamma)$ to
(\ref{Poisson equ})-(\ref{V(u)-V(U)}). The solution $V$ is unique up
to a constant. In the following proposition,
 let
\begin{eqnarray}\label{eq:lambda}
\lambda=2\mu/\sigma^2.
\end{eqnarray}
\begin{proposition}
\label{prop:controlband}
 Let $\varphi=\{d,D,U,u\}$ be a control band policy with
  \begin{displaymath}
    d<D<U<u.
  \end{displaymath}
Let  $m \in \R$ be any fixed number. Define
\begin{displaymath}
  V(x)=  \int_{m}^x g(y) dy
\end{displaymath}
with
\begin{equation}
  \label{eq:g}
  g(x)= V'(m) e^{\lambda (m-x)}  + \gamma \frac{2}{\sigma^2} \int_m^x
e^{\lambda (y-x)}dy -\frac{2}{\sigma^2}\int_m^x h(y)e^{\lambda (y-x)}dy,
\end{equation}
where
\begin{eqnarray}
 &&  \gamma =  \frac{a_1\bigl(c_2+L+\ell(u-U)\bigr)+a_2\bigl(c_1 +
    K+k(D-d)\bigr)}{a_2b_1 +
    a_1b_2}, \label{eq:gamma} \\
 && V'(m) =  \frac{b_1\bigl(c_2+L+\ell(u-U)\bigr)-b_2\bigl(c_1 +
    K+k(D-d)\bigr)}{a_2b_1 +
    a_1b_2}. \label{eq:Vprime}
\end{eqnarray}
Then $(V, \gamma)$ is a solution to
(\ref{Poisson    equ})-(\ref{V(u)-V(U)}). In (\ref{eq:gamma}) and
(\ref{eq:Vprime}), we set
\begin{eqnarray}
&&  a_1 =  \int_d^D e^{\lambda(m-x)}dx,  \quad
  a_2 =  \int_U^u e^{\lambda(m-x)}dx, \label{eq:coeffa}
\\
&& b_1=-\frac{2}{\sigma^2}
\int_d^D \int_m^x
e^{\lambda (y-x)}dy dx, \quad
 b_2=\frac{2}{\sigma^2}
\int_U^u \int_m^x
e^{\lambda (y-x)}dy dx, \label{eq:coeffb}
\\
&& c_1=-\frac{2}{\sigma^2}\int_d^D\int_m^x h(y)e^{\lambda
  (y-x)}dydx, \quad
c_2=\frac{2}{\sigma^2}\int_U^u\int_m^x h(y)e^{\lambda
  (y-x)}dydx.\label{eq:coeffc}
\end{eqnarray}
\end{proposition}
\begin{proof}
Equation (\ref{Poisson equ})
is  equivalent to
\begin{displaymath}
  (e^{\lambda x}V'(x))' = \frac{2}{\sigma^2}(\gamma-h(x))e^{\lambda x}.
\end{displaymath}
Integrating over $[m, x]$ on both sides, we have
\begin{eqnarray*}
e^{\lambda x}V'(x)&=&e^{\lambda m}  V'(m) + \gamma \frac{2}{\sigma^2} \int_m^x
e^{\lambda y}dy -\frac{2}{\sigma^2}\int_m^x h(y)e^{\lambda y}dy
\end{eqnarray*}
or equivalently
\begin{displaymath}
  V'(x)=
e^{\lambda (m-x)}  V'(m) + \gamma \frac{2}{\sigma^2} \int_m^x
e^{\lambda (y-x)}dy -\frac{2}{\sigma^2}\int_m^x h(y)e^{\lambda (y-x)}dy.
\end{displaymath}
Boundary conditions (\ref{V(d)-V(D)}) and (\ref{V(u)-V(U)}) become
\begin{eqnarray}
\lefteqn{  V'(m) \int_d^D e^{\lambda(m-x)}dx +  \gamma \frac{2}{\sigma^2}
\int_d^D \int_m^x
e^{\lambda (y-x)}dy dx }\nonumber \\
&=& \frac{2}{\sigma^2}\int_d^D\int_m^x h(y)e^{\lambda
  (y-x)}dydx  -K-k(D-d),\label{eq:boundaryD} \\
\lefteqn{ V'(m) \int_U^u e^{\lambda(m-x)}dx +  \gamma \frac{2}{\sigma^2}
\int_U^u \int_m^x
e^{\lambda (y-x)}dy dx }\nonumber \\
&=&\frac{2}{\sigma^2}\int_U^u\int_m^x h(y)e^{\lambda
  (y-x)}dydx + L+\ell(u-U).\label{eq:boundaryU}
\end{eqnarray}
Using the coefficients defined in (\ref{eq:coeffa})-(\ref{eq:coeffc}),
we see the
 boundary conditions (\ref{eq:boundaryD}) and
(\ref{eq:boundaryU}) become
\begin{eqnarray*}
&&  a_1 V'(m) -\gamma b_1 = -(c_1+K+k(D-d)), \\
&&  a_2 V'(m) + \gamma b_2 = c_2 + L + \ell (u-U),
\end{eqnarray*}
from which we have unique solution for $\gamma$ and  $V'(m)$
given in (\ref{eq:gamma}) and (\ref{eq:Vprime}).
\end{proof}

\subsection{Optimal Policy and Optimal Parameters}
\label{sec:optimal}
Theorem \ref{thm:lowerbound} suggests the following strategy to obtain
an optimal policy. We hope that a control band
policy is optimal. Therefore, the first task is to find an optimal
policy among all control band policies. We denote this optimal control
band policy by $\varphi^*=\{d^*, D^*, U^*, u^*\}$ with long-run average cost
$\gamma^*$.  We hope that $\gamma^*$ can be used as the constant in
\eqref{eq:lbPoission} of Theorem \ref{thm:lowerbound}. To find the
corresponding $f$ that, together  with the $\gamma^*$,  satisfies all
the conditions of  Theorem
\ref{thm:lowerbound}, we start with the relative
value function
$V(x)$ associated with the policy $\varphi^*$. This relative value
function $V$  is defined on the finite
interval $[d^*, u^*]$. We need to extend $V$ so that it is
defined on the entire real line $\R$. Given that $V(x)$ is the
relative value function, it is natural to extend it in the following way
\begin{equation}
  \label{eq:f}
  f(x) =
  \begin{cases}
   K + k(D^*-x) + V(D^*) & \text{ for } x< d^*, \\
    V(x) & \text{ for } x\in [d^*, u^*], \\
   L + \ell(x-U^*) + V(U^*) & \text{ for } x> u^*.
  \end{cases}
\end{equation}
Boundary conditions (\ref{V(d)-V(D)}) and (\ref{V(u)-V(U)}) ensure the
continuity of $f$ at $d^*$ and $u^*$. Therefore, $f\in C(\R)$.
We are yet to determine the optimal parameters $(d^*, D^*, U^*,
u^*)$. Now we provide an intuitive argument on the conditions that
should be  imposed
on the optimal parameters.
Since we wish $f\in \mathbb{C}^1$, we should have
\begin{equation}
  \label{eq:BoundaryK}
  V'(d^*)=-k, \quad  V'(u^*)=\ell.
\end{equation}
Also, starting from $d^*$, the system should jump to a $D$ that
minimizes
\begin{displaymath}
K+   k(D-d^*) + V(D).
\end{displaymath}
Therefore, at $D=D^*$, $k+V'(D)=0$, namely,
\begin{equation}
  \label{eq:boundary2}
  V'(D^*)=-k.
\end{equation}
Similarly, one should have
\begin{equation}
  \label{eq:boundary3}
  V'(U^*)=\ell.
\end{equation}

In this section, we will first prove in
Theorem~\ref{lem:optimalParameters} the existence of parameters $d^*$,
$D^*$, $U^*$ and $u^*$ such that the relative value function $V$
corresponding the control band policy $\varphi=\{d^*, D^*, U^*, u^*\}$
satisfies (\ref{Poisson equ})-(\ref{V(u)-V(U)}), and
(\ref{eq:BoundaryK})-(\ref{eq:boundary3}).  As part of the solution,
we are to  find the boundary points $d^*$, $D^*$, $U^*$ and $u^*$ from
equations (\ref{Poisson equ})-(\ref{V(u)-V(U)}) and
(\ref{eq:BoundaryK})-(\ref{eq:boundary3}). These equations define a
\emph{free boundary problem}.  The solution to a free boundary problem
is much more
difficult to be found than the one to a boundary value problem.  We
then prove in Theorem~\ref{thm:optimal} that the extension $f$ in
(\ref{eq:f}) and $\gamma^*=\AC(\varphi^*, x)$ jointly satisfy all the
conditions in Theorem \ref{thm:lowerbound}; therefore, the control
band policy  $\varphi^*$ is optimal among all feasible policies.

To ease the presentation, in the rest of this section, we assume that
$\mu>0$.  The statement and  analysis for the cases $\mu<0$ and
$\mu=0$ are analogous and are omitted.

To facilitate the presentation of Theorem \ref{lem:optimalParameters}, we
first find a general solution $V$ to (\ref{Poisson equ}) without
worrying about boundary conditions (\ref{V(d)-V(D)}) and
(\ref{V(u)-V(U)}). Proposition \ref{prop:controlband} shows that such $V$ is
given in the form
\begin{equation}
  \label{eq:integralV}
 V(x)=\int_m^xg(y)y \quad \text{ for } x\in [d^*, u^*],
\end{equation}
 where $g$ is given by (\ref{eq:g}) and $m$ is some constant. Since
 the optimal boundary points $d^*, D^*, U^*,
u^*$ are  yet to be determined,
 the constant $\gamma$ on the right side of (\ref{Poisson equ}) is also
 yet to be determined. Differentiating both sides of (\ref{Poisson
   equ}) with respect to $x$, we have shown that $V'(x)=g(x)$ is a solution to
  \begin{equation}
    \Gamma g(x) + h'(x) = 0 \quad \text{ for all } x\in \R\setminus\{a\},
  \label{eq:gPoisson}
  \end{equation}
In (\ref{eq:g}), we fix $m=a$ and set $A=2\gamma/(\lambda\sigma^2)$
and $B=A-V'(m)$.  Noting that $\lambda/\mu=\frac{2}{\sigma^2}$,
we have $g(x)=g_{A, B}(x)$, where
\begin{eqnarray} \nonumber
g_{A, B}(x)&=& A - B e^{-\lambda(x-a)} -(\lambda/\mu) \int_a^x h(y)  e^{-\lambda(x-y)}dy \\
&=&A-Be^{-\lambda (x-a)}-\frac{\lambda}{\mu}\int_a^x h(x-y+a)
e^{-\lambda (y-a)} dy.
    \label{eq:g1}
\end{eqnarray}
To summarize, we have the following lemma.
\begin{lemma}\label{lem:solutionToPoisson}
  For each $A, B\in \R$, function $g(x)=g_{A, B}(x)$ is a solution to
  equation (\ref{eq:gPoisson}).
\end{lemma}


\begin{figure}[t]
  \centering
  \includegraphics[width=10cm]{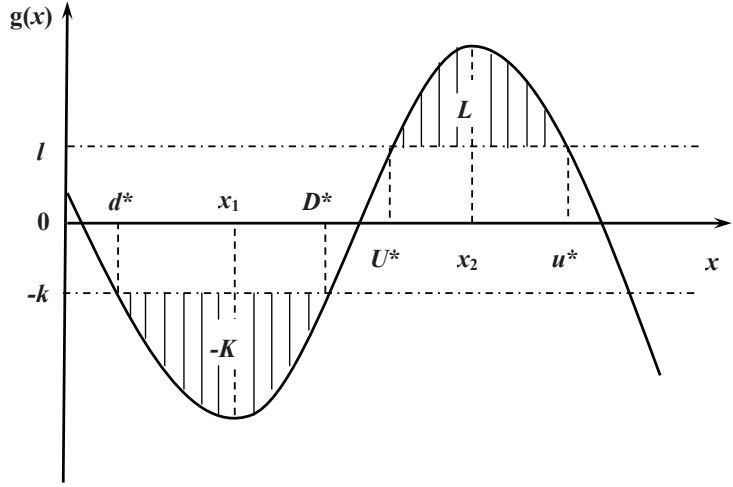}
\caption{ There exist $x_1<x_2$ such that  the function $g$  decreases in
  $(-\infty, x_1)$, increases in $(x_1, x_2)$, and deceases again in $(x_2,
  \infty)$. Parameters $d^*$, $D^*$, $U^*$ and $u^*$ are determined by
$g(d^*)=g(D^*)=-k$, $g(U^*)=g(u^*)=\ell$, the shaded area between
$U^*$ and $u^*$ is $L$, and the shaded area between $d^*$ and $D^*$ is
$K$. In the interval $[d^*, u^*]$, $g$ is the derivative of the
relative value function associated with the control band policy
$\{d^*, D^*, U^*, u^*\}$. }
\label{fig:g1}
\end{figure}
The following theorem characterizes
optimal parameters $(d^*, D^*, U^*, u^*)$ via solution $g=g_{A, B}$. Figure
\ref{fig:g1} depicts the function $g$ used
in the theorem.
\begin{theorem}\label{lem:optimalParameters}
  Assume that the holding cost function $h$ satisfies Assumption
  \ref{assumption:h}.
There exist unique $A^*$, $B^*$, $d^*$, $D^*$, $U^*$ and $u^*$
with
  \begin{displaymath}
    d^*< x_1< D^* < U^*<x_2< u^*
  \end{displaymath}
such that the corresponding $g(x)=g_{A^*, B^*}(x)$ satisfies
\begin{eqnarray}
  && \int_{d^*}^{D^*} [g(x)+k] dx = -K, \label{eq:gK}\\
  && \int_{U^*}^{u^*}[g(x)-\ell] dx = L, \label{eq:gL}\\
  && g(d^*)=g(D^*) =-k, \label{eq:gk}\\
  && g(U^*)=g(u^*) =\ell.\label{eq:gell}
\end{eqnarray}
Furthermore, $g$ has a local minimum at $x_1<a$ and a local maximum at
$x_2>a$. The function $g$ is decreasing on $(-\infty, x_1)$, increasing
on $(x_1, x_2)$ and decreasing again on $(x_2, \infty)$.
\end{theorem}
If $g$ satisfies all conditions (\ref{eq:gPoisson}),
(\ref{eq:gK})-(\ref{eq:gell}) in Theorem~\ref{lem:optimalParameters},
$V(x)$ in (\ref{eq:integralV})  clearly satisfies all conditions
(\ref{Poisson equ})-(\ref{V(u)-V(U)}) and
(\ref{eq:BoundaryK})-(\ref{eq:boundary3}).
The proof of Theorem~\ref{lem:optimalParameters} is long, and we
defer  it to end of this section.

\begin{theorem}\label{thm:optimal}
  Assume that the holding cost function $h$ satisfies Assumption
  \ref{assumption:h}. Let $d^*<D^*<U^*<u^*$, along with constants $A^*$ and
  $B^*$, be the unique solution in
  Theorem~\ref{lem:optimalParameters}. Then the control band policy
  $\varphi^*=\{d^*, D^*, U^*, u^*\}$ is optimal among all
 feasible policies.
\end{theorem}
\begin{proof}
Let $g(x)$ be the function in (\ref{eq:g1}) with $A=A^*$
and $B=B^*$. Let
\begin{displaymath}
  \bar g(x) =
  \begin{cases}
    -k, & x< d^*, \\
    g(x), & d^*\le x \le u^*, \\
    \ell, & x \ge u^*.
  \end{cases}
\end{displaymath}
Conditions (\ref{eq:gk}) and (\ref{eq:gell}) ensure that $\bar{g}$ is $C(\R)$.
Define
\begin{displaymath}
  V(x) = \int_{d^*}^x \bar g(y)dy.
\end{displaymath}
Let $\gamma^*$ be the long-run average cost under policy
$\varphi^*$. We now show that $V$ and $\gamma^*$ satisfy all the
conditions in Theorem \ref{thm:lowerbound}. Thus, Theorem
\ref{thm:lowerbound} shows that the long-run average cost under any
feasible
policy is at least $\gamma^*$. Since $\gamma^*$ is the long-run
average cost under the control band policy $\varphi^*$, $\gamma^*$ is
the optimal cost and the control band policy $\varphi^*$ is optimal
among all feasible policies.

First,  $V(x)$ is in $C^2((d^*, u^*))$. Condition \eqref{eq:gK}
implies
\begin{equation}
  \label{eq:boundaryStarK}
  V(d^*)-V(D^*) = K  + k(D^*-d^*)
\end{equation}
and \eqref{eq:gL} implies
\begin{displaymath}
  V(u^*)-V(U^*) = L  + \ell(u^*-U^*).
\end{displaymath}
Equation \eqref{eq:gPoisson} implies that $V$ satisfies
\begin{displaymath}
 \Gamma V + h(x) = \text{ constant for }  x\in (d^*, u^*).
\end{displaymath}
By Theorem \ref{thm:control band}, the constant must be the long-run
average cost $\gamma^*$ under control band policy $\varphi^*$.

Now, we show that $V(x)$ satisfies the rest of conditions in Theorem
\ref{thm:lowerbound}. Conditions \eqref{eq:gk} and \eqref{eq:gell}
imply that truncated function $\bar g$ is continuous in
$\R$. Therefore, $V\in \mathbb{C}^1(\R)$. Clearly, $V''(x)=0$ for $x\not \in
[d^*, u^*]$, and  $V''(x)=g'(x)$ for $x\in (d^*, u^*)$. Let
\begin{displaymath}
  M = \sup_{x\in [d^*, u^*]} \abs{g(x)}.
\end{displaymath}
We have $\abs{V'(x)}\le M$ for all $x\in \R$.
Because
\begin{displaymath}
 \Gamma V+h(x) = \gamma^* \quad \text{ for } x\in (d^*, u^*),
\end{displaymath}
\eqref{eq:lbPoission} is satisfied for $x\in (d^*, u^*)$.
In particular
\begin{displaymath}
\frac{1}{2}\sigma^2  g'(d^*)+\mu g(d^*)+h(d^*)=\gamma^*
\end{displaymath}
and
\begin{displaymath}
\frac{1}{2}\sigma^2  g'(u^*)+\mu g(u^*)+h(u^*)=\gamma^*.
\end{displaymath}
It follows from  part (b) and part (c)  of
Lemma~\ref{lem:optimalParameters} in Section
\ref{sec:optimal-control-band} that $d^*<x_1<a<x_2<u^*$,
$g'(d^*)\le 0$  and $g'(u^*)\le 0$ (see Figure \ref{fig:g1}).
Thus, we have $\mu g(d^*)+h(d^*)\ge \gamma^*$
and $\mu g(u^*)+h(u^*)\ge\gamma^*$. Now,
for $x<d^*$, $\Gamma V(x)+h(x)=\mu (-k)+h(x)\ge \mu g(d^*)+h(d^*)\ge\gamma^*$. Similarly, for
$x>u^*$, $\Gamma V(x)+h(x)=\mu(\ell)+h(x)\ge\mu g(u^*)+h(u^*)\ge \gamma^*$.

Now we verify that $V$ satisfies \eqref{eq:lbK}.
Let $x, y\in \R$ with $y<x$.
Then,
\begin{eqnarray*}
  V(x)-V(y) + k(x-y) &= & \int_y^x[\bar g(z)+k]dz \\
    &\ge &  \int_{(y\vee d^*)\wedge D^*}^{(x\wedge D^*)\vee d^*}[\bar g(z)+k]dz \\
    &\ge &  \int_{d^*}^{D^*}[\bar g(z)+k]dz \\
  &=& -K,
\end{eqnarray*}
where the first inequality follows from $\bar g(z)=g(z)=-k$ for $z\le
d^*$ and  $\bar g(z)=g(z)\ge -k$ for $D^*< z<
u^*$ and $\bar g(z)=\ell \ge -k$ for $z\ge u^*$,
 and the second inequality follows from the fact that $\bar
 g(z)=g(z)\le -k$ for
$z\in [d^*, D^*]$; see, Figure \ref{fig:g1}. Thus \eqref{eq:lbK} is
proved.

It remains to verify that $V$ satisfies \eqref{eq:lbL}.
For $x, y\in \R$ with $y>x$.
\begin{eqnarray*}
  V(y)-V(x) -\ell (y-x) &= & \int_x^y[\bar g(z)-\ell]dz \\
    &\le &  \int_{(x\vee U^*)\wedge u^*}^{(y\wedge u^*)\vee U^*}[\bar g(z)-\ell]dz \\
    &\le &  \int_{U^*}^{u^*}[\bar g(z)-\ell]dz \\
  &=& L,
\end{eqnarray*}
proving \eqref{eq:lbL}.
\end{proof}

\subsection{Optimal Control Band Parameters}
\label{sec:optimal-control-band}
This section is devoted to the proof of Theorem
\ref{lem:optimalParameters}.  We separate the proof into a series of
lemmas. Throughput of this section,
we assume that $\mu>0$ and that the holding cost function $h$ satisfies
Assumption \ref{assumption:h}. Recall the $\lambda$ defined in
(\ref{eq:lambda}).

Define
\begin{eqnarray}\label{eq:under b over B}
\overline{B}=-\frac{1}{\mu}\int_{-\infty}^a
h'(y)e^{\lambda(y-a)}dy.
\end{eqnarray}
Because $h'(x)< 0$ for $x<a$, $\overline{B}>0$.
For $A, B\in \R$, recall the function $g_{A, B}$ defined in
(\ref{eq:g1}). We sometime use the fact that
\begin{equation}
  \label{eq:gABg0B}
  g_{A, B}(x)= A + g_{0, B}(x) \quad  \text{ for } x\in \R.
\end{equation}
 When the context is clear, we simply use $g$ to denote $g_{A, B}$.
For the following lemma, readers are referred to Figure \ref{fig:g1}.
\begin{lemma}
  \label{thm:optimalParameter}
  (a) For any $A\in \R$ and for each fixed
  $B\in(0,\overline{B})$, $g_{A, B}$ attains a unique
  minimum in $(-\infty, a)$ at $x_1=x_1(B)\in (-\infty, a)$. The
  function $g_{A, B}$ attains a unique maximum in $(a, \infty)$ at
  $x_2=x_2(B) \in (a, \infty)$. Both $x_1(B)$ and $x_2(B)$  are
  independent of $A$.

  (b) For each fixed $B\in(0,\overline{B})$, the local
  minimizer $x_1=x_1(B)$ is the unique solution in $(-\infty, a)$ to
  \begin{equation}
    \label{eq:xiBequation}
    B-\frac{1}{\mu}\int_a^x h'(y)  e^{\lambda (y-a)} dy=0.
  \end{equation}
  The local maximizer $x_2=x_2(B)$ is the unique solution in $(a,
  \infty)$ to (\ref{eq:xiBequation}).

  (c) For each $B\in(0,\overline{B})$, $g_{A, B}'(x)<0$
  for $x\in (-\infty, x_1(B))$, $g_{A, B}'(x)>0$ for $x\in (x_1(B), x_2(B))$, and
$g_{A, B}'(x)<0$ for $x\in(x_2(B), \infty)$.
\end{lemma}

\begin{proof}
Differentiating $g(x)=g_{A, B}(x)$ in (\ref{eq:g1}) and noting $h(a)=0$, we have
\begin{eqnarray}
  g'(x)&=&\lambda Be^{-\lambda (x-a)}-\frac{\lambda}{\mu}\int_a^x h'(x-y+a)  e^{-\lambda (y-a)} dy
 \label{eq:g'}\\
 &=&  \lambda\Big(B-\frac{1}{\mu}\int_a^x h'(y)
 e^{\lambda (y-a)} dy\Big)e^{-\lambda (x-a)} \nonumber\\
&=& \lambda F_1(B,x) e^{-\lambda (x-a)}, \nonumber
\end{eqnarray}
where, for $x\in \R$,
\begin{equation}
  \label{eq:F1}
F_1(B,x)= B-\frac{1}{\mu}\int_a^x h'(y)  e^{\lambda (y-a)} dy.
\end{equation}
Clearly $g'(x)=0$ if and only if $F_1(B,x)=0$. Because
\begin{eqnarray}
\label{eq:s2}
\frac{\partial}{\partial x}F_1(B,x)=-\frac{1}{\mu}h'(x)e^{\lambda(x-a)}
\end{eqnarray}
and $h'(x)<0$ for $x<a$ and $h'(x)>0$ for $x>a$, we have
that $F_1(B,x)$ increases in $x<a$ and decreases in $x>a$.
For  $B>0$, we have
\begin{displaymath}
F_1(B, a)=B>0.
\end{displaymath}
For any  $B\in(0, \overline{B})$,
\begin{eqnarray*}
\lim_{x\downarrow -\infty}F_1(B,x)=B-\overline{B}<0.
\end{eqnarray*}
Therefore, there exists a unique
$x_1=x_1(B)\in(-\infty,a)$ such that $F_1(B, x_1)=0$ or equivalently
$g'(x_1)=0$.
Also, for any fixed $B$
\begin{eqnarray*}
\lim_{x\uparrow +\infty}F_1(B,x)=-\infty.
\end{eqnarray*}
Therefore, for any $B> 0$,
there exists a unique $x_2=x_2(B)\in (a, \infty)$ such that
$F_1(B,x_2)=0$ or equivalently $g'(x_2)=0$. For $B\in (0, \overline{B})$, it is clear that
\begin{eqnarray*}
g'(x)<0 \text{ for } x\in (-\infty, x_1), \quad
 g'(x)>0 \text{ for } x\in(x_1, x_2) \quad \text{and} \quad
g'(x) <0 \text{ for } x\in (x_2, \infty).
\end{eqnarray*}
Thus the lemma is proved.
\end{proof}
\begin{remark}
  The local maximizer $x_2(B)$ is well defined for all
$B\in(0,\infty)$, whereas the local minimizer
$x_1(B)$ is defined only for $B\in(0, \overline{B})$.
\end{remark}

\begin{lemma}
\label{lem:x_i}
(a) The local minimizer  $x_1(B)$  is continuous and strictly
decreasing in $B\in(0, \overline{B})$.
 The local maximizer $x_2(B)$  is continuous and strictly
increasing  in $B\in(0, \infty)$.
Furthermore,
 \begin{equation}
   \label{eq:x1Blimit}
 \lim_{B\downarrow 0}  x_i(B)=a \quad i=1, 2
 \end{equation}
and
\begin{equation}
  \label{eq:x2Blimit}
 \lim_{B\uparrow \overline{B}}  x_1(B) =-\infty \quad \text{ and } \quad
 \lim_{B\uparrow \overline{B}}  x_2(B)=x_2(\overline{B})\in (a, \infty).
\end{equation}
(b) For each $B\in(0, \overline{B})$,
\begin{equation}
  \label{eq:gh}
  g_{A,B}(x_i(B))= A - \frac{1}{\mu}h(x_i(B)) \text{ for } i =1,2.
\end{equation}
\end{lemma}
\begin{proof}
(a) Recall the function $F_1$ defined in (\ref{eq:F1}).
Obviously, $F_1$, $\frac{\partial F_1}{\partial B}$,  $\frac{\partial F_1}{\partial x}$ are continuous,
and $\frac{\partial F_1}{\partial  x}$ is given in (\ref{eq:s2}). One has
\begin{displaymath}
\frac{\partial F_1}{\partial
  x} > 0 \quad \text{ for }
x \in(-\infty, a),
\end{displaymath}
where we have used the fact that $h'(x)<0$ for $x\in (-\infty, a)$.
Using the {\em Implicit Function Theorem}, $x_1(B)$ is continuously
differential in $B\in(0, \overline{B})$, and
\begin{eqnarray}
\frac{d x_1(B)}{d
  B}=\frac{\mu}{h'(x_1(B))e^{\lambda(x_1(B)-a)}}<0. \label{eq:dx_1/dB}
\end{eqnarray}
Thus,  $x_1(B)$ is strictly decreasing in $B\in(0, \overline{B})$. Similarly, we have
\begin{eqnarray}
\frac{d x_2(B)}{d B}=\frac{\mu}{h'(x_2(B))e^{\lambda(x_2(B)-a)}}>0 \label{eq:dx_2/dB}
\end{eqnarray}
 proving that $x_2(B)$ continuously differential and strictly
 increasing  in $B\in(0, \infty)$.
 The limits in (\ref{eq:x1Blimit}) and (\ref{eq:x2Blimit}) can be
 proved easily following the definition of $x_1(B)$  and $x_2(B)$.

(b)  We have from (\ref{eq:g1}) and (\ref{eq:xiBequation}) that
\begin{eqnarray*}
&&g_{A,B}(x_i(B))\nonumber\\
&&\ \ \ =A-Be^{-\lambda (x_i(B)-a)}-\frac{\lambda}{\mu}\int_a^{x_i(B)} h(x_i(B)-y+a)  e^{-\lambda (y-a)} dy\nonumber\\
&&\ \ \ =A-\frac{1}{\mu}\int_a^{x_i(B)} h'(x_i(B)-y+a)e^{-\lambda (y-a)} dy
-\frac{\lambda}{\mu}\int_a^{x_i(B)} h(x_i(B)-y+a)  e^{-\lambda (y-a)} dy\nonumber\\
&&\ \ \ =A-\frac{\lambda}{\mu}\int_a^{x_i(B)} h(x_i(B)-y+a)  e^{-\lambda (y-a)} dy\nonumber\\
&&\ \ \ \ \ +\frac{1}{\mu}\Big[ h(x_i(B)-y+a)e^{-\lambda (y-a)}\mid_a^{x_i(B)}+\lambda \int_a^{x_i(B)}
h(x_i(B)-y+a)e^{-\lambda (y-a)} dy \Big]\nonumber\\
&&\ \ \ =A-\frac{1}{\mu}h(x_i(B)), \label{eq:g(x_i)}
\end{eqnarray*}
thus proving (\ref{eq:gh}).
\end{proof}
In the following lemma, we set for each $B\in (0,\infty)$,
\begin{equation}
  \label{eq:AunderlineB}
\underline{A}(B)=\ell-g_{0, B}(x_2(B))= h(x_2(B))/\mu+\ell.
\end{equation}
For any $B\in (0,\infty)$,  following (\ref{eq:gABg0B}), we
have
\begin{equation}
  \label{eq:gbiggerell}
  g_{A, B}(x_2(B))\ge  \ell
\end{equation}
for any $A \ge \underline{A}(B)$. Similarly, for any
$B\in(0, \overline{B})$, we define
\begin{equation}
  \label{eq:AoverlineB}
\overline{A}(B)=-k-g_{0, B}(x_1(B))=h(x_1(B))/\mu-k.
\end{equation}
Following (\ref{eq:gABg0B}), we
have
\begin{equation}
  \label{eq:gless-k}
  g_{A, B}(x_1(B))\le -k
\end{equation}
for any $A\le \overline{A}(B)$. Our next lemma determines when
$\underline{A}(B)< \overline{A}(B)$.

\begin{lemma} \label{lem:B1}
 For each $B\in(0, \overline{B})$,  let
 \begin{equation}
   \label{eq:tildeg}
   \tilde g(B)=g_{A, B}(x_2(B))-g_{A, B}(x_1(B))
 \end{equation}
 be the distance between the
local maximum and the local minimum. Then $\tilde g(B)$ is independent
of $A$. The function $\tilde g(B)$  is continuous and  strictly
increasing in
$B\in(0, \overline{B})$ with
\begin{eqnarray}
 \lim_{B\downarrow 0}\widetilde{g}(B)=0\quad \text{ and } \quad
 \lim_{B\uparrow \overline{B}}\widetilde{g}(B)=+\infty. \label{eq:tildeg0infinity}
\end{eqnarray}
Thus, there exists a unique
$\underline{B}_1\in(0, \overline{B})$ such that
\begin{equation}
  \label{eq:B1}
  \tilde g(\underline{B}_1) = k +\ell.
\end{equation}
For each $B\in (\underline{B}_1, \overline{B})$,
\begin{equation}
  \label{eq:underlineAlessoverlineA}
  \underline{A}(B) < \overline{A}(B).
\end{equation}
\end{lemma}
\begin{proof}
 By (\ref{eq:gABg0B}), $\tilde g(B)= g_{0,B}(x_2(B))-
g_{0, B}(x_1(B))$. Thus, $\tilde g(B)$ is independent of $A$.
It follows from (\ref{eq:dx_1/dB}) that
for  $B\in(0, \overline{B})$
\begin{eqnarray}
\frac{d \tilde{g}(B)}{d B}
&=&
-\frac{1}{\mu}\Big[h'(x_2(B))\frac{dx_2(B)}{dB}-h'(x_1(B))\frac{dx_1(B)}{dB}\Big]
\nonumber \\
&=&  -e^{-\lambda(x_2(B)-a)}+e^{-\lambda(x_1(B)-a)} \label{eq:d wide(g)/dB} \\
\nonumber
&>&0.\nonumber
\end{eqnarray}
Thus $\tilde g(B)$ is strictly increasing. The limit
(\ref{eq:tildeg0infinity}) follows from (\ref{eq:x1Blimit}) and
(\ref{eq:x2Blimit}). The existence of unique $\underline{B}_1$
satisfying (\ref{eq:B1}) follows from (\ref{eq:tildeg0infinity}),
the continuity and monotonicity of $\tilde g$. Inequality
(\ref{eq:underlineAlessoverlineA}) follows from the definition of
$\underline{B}_1$ and the fact that $
\overline{A}(B)-\underline{A}(B)=\tilde{g}(B)-(\ell+k)$.

\end{proof}
\begin{lemma}
  \label{lem:Uu}
(a) For each $B\in (\underline{B}_1, \overline{B})$ and each $A\in
\bigl(\underline{A}(B), \overline{A}(B)]$, there exist unique $U(A,
B)$ and $u(A,
B)$ with
\begin{equation}
  \label{eq:UABlessuAB}
 x_1(B) < U(A, B) < x_2(B) < u(A, B)
\end{equation}
such that
\begin{eqnarray}
&&  g_{A, B}(U(A, B))=g_{A, B}(u(A, B))=\ell, \label{eq:Uudef} \\
&& g'_{A, B}(U(A,B))>0, \quad g'_{A, B}(u(A,
B))<0, \label{eq:UuDerivative} \\
&& g_{A, B} (x_1(B))\le -k. \label{eq:gx1Bless-k}
\end{eqnarray}
(b) For each fixed $B\in (\underline{B}_1, \overline{B})$, $U(A, B)$  and
$u(A, B)$ are continuous differentiable function in
$A\in \bigl(\underline{A}(B), \overline{A}(B)\bigr)$. The function $U(A,B)$ is
decreasing in $A$ and the function $u(A, B)$ is increasing in $A$.
\end{lemma}
\begin{proof}
  (a) For each $B\in (\underline{B}_1, \overline{B})$ and each $A\in
  \bigl(\underline{A}(B), \overline{A}(B)]$, we have $g_{A,
    B}(x_2(B))>\ell$ and $g_{A,B}(x_1(B))\leq -k$. Thus, there are unique
  $U(A, B)$ and $u(A, B)$ that satisfy
  (\ref{eq:Uudef})-(\ref{eq:UuDerivative}).  When
  $A\in\bigl(\underline{A}(B), \overline{A}(B))$, the inequality
  (\ref{eq:gx1Bless-k}) holds.  This inequality implies that $U(A,
  B)>x_1(B)$, which in turn implies that inequality
  (\ref{eq:UABlessuAB}) holds.

(b) Using the \emph{Implicit Function Theorem}, we have
\begin{eqnarray*}
&&  \frac{\partial}{\partial A} U(A, B) = - \frac{1}{g'_{A, B}(U(A, B))}
  <0, \\
&&  \frac{\partial}{\partial A} u(A, B) = - \frac{1}{g'_{A, B}(u(A, B))}
  >0.
\end{eqnarray*}
This proves part (b) of the lemma.
\end{proof}

Fix a $B\in (\underline{B}_1, \overline{B})$. For $A\in
\bigl(\underline{A}(B), \overline{A}(B)\bigr)$ let
\begin{equation}
  \label{eq:Lambda_2AB}
  \Lambda_2(A, B) = \int_{u(A, B)}^{U(A, B)} \bigl[ g_{A, B}(x) - \ell\bigr] dx.
\end{equation}
We would like to show that there exists a unique $A^*(B)\in
(\underline{A}(B),\overline{A}(B)) $ such that
\begin{equation}
  \label{eq:Lambda2AstarEqualL}
\Lambda_2(A^*(B),B)=L.
\end{equation}
\begin{lemma}\label{lem:Lambda2continuity}
Fix a $B\in (\underline{B}_1, \overline{B})$. The function $\Lambda_2(A,
B)$ is continuous and strictly increasing in
$A\in \bigl(\underline{A}(B), \overline{A}(B)\bigr)$. Furthermore
\begin{equation}
  \label{eq:Lambda2underlineAB0}
\lim_{A\downarrow \underline{A}(B) }  \Lambda_2(A, B) =0.
\end{equation}
\end{lemma}
\begin{proof}
By the {\em Implicit Function Theorem}, we have
\begin{eqnarray}
\frac{\partial \Lambda_2 (A,B)}{\partial A}&=&
\frac{\partial u(A,B)}{\partial A} \bigl[g_{A,B}(u(A,B)-l)\bigr]
-\frac{\partial U(A,B)}{\partial A} \bigl[g_{A,B}(U(A,B)-l)\bigr]\nonumber\\
&&{}+\int_{U(A,B)}^{u(A,B)} 1 dx\nonumber\\
&=&u(A,B)-U(A,B)\nonumber\\
&>&0.\label{eq:d Lambda_1/dA}
\end{eqnarray}
Therefore $\Lambda_2(A,B)$ is strictly increasing in
$A\in(\underline{A}(B),\overline{A}(B))$.

Observe that
\begin{displaymath}
\lim_{A\downarrow \underline{A}(B)}g_{A,B}(x_2(B))=\lim_{A\downarrow
    \underline{A}(B)}\bigl[A-\frac{1}{\mu}h(x_2(B))\bigr]=l.
\end{displaymath}
By the definitions of $U(A,B)$ and $u(A,B)$, we have
\begin{displaymath}
\lim_{A\downarrow \underline{A}(B)}U(A,B)=\lim_{A\downarrow
\underline{A}(B)}u(A,B)=x_2(B),
\end{displaymath}
which proves (\ref{eq:Lambda2underlineAB0}).

\end{proof}

\begin{lemma} \label{lem:B1B2}
The function $\Lambda_2(\overline{A}(B), B)$ is continuous and
strictly increasing in  $B\in (\underline{B}_1, \overline{B})$.
Furthermore,
\begin{equation}
  \label{eq:1}
\lim_{B\downarrow \underline{B}_1}  \Lambda_2(\overline{A}(B), B)=0
\quad\text{and} \quad
\lim_{B\uparrow \overline{B}}  \Lambda_2(\overline{A}(B), B)=\infty.
\end{equation}
Therefore,
 there exists a unique
$\underline{B}_2\in (\underline{B}_1, \overline{B})$ such that
\begin{equation}
  \label{eq:B2}
\Lambda_2(\overline{A}(\underline{B}_2), \underline{B}_2) = L \quad \text{and}
\quad
\Lambda_2(\overline{A}(B),B) > L \quad \text{for }
B\in (\underline{B}_2, \overline{B}).
\end{equation}
\end{lemma}
\begin{proof}

  We first check that $\Lambda_2(\overline{A}(B), B)$ is strictly increasing in $B\in (\underline{B}_1, \overline{B})$.
  To see this, it follows from
 (\ref{eq:dx_1/dB})  and  the definition of
 $\overline{A}(B)$ in (\ref{eq:AoverlineB}) that
 \begin{equation}
\frac{d\overline{A}(B)}{dB}=\frac{1}{\mu}h'(x_1(B))\frac{dx_1(B)}{dB}
=e^{-\lambda(x_1(B)-a)},\label{eq:d(A)/dB}
 \end{equation}
which implies
\begin{eqnarray}
&&\frac{d \Lambda_2(\overline{A}(B),B)}{d B}  \label{eq:dLamma/da}\\
&&\ \ \ =\frac{\partial u(\overline{A}(B),B)}{\partial B} [g_{\overline{A}(B),B}(u(A,B)-l)]
+\frac{\partial u(\overline{A}(B),B)}{\partial A}\frac{d\overline{A}(B)}{dB} [g_{\overline{A}(B),B}(u(A,B)-l)]\nonumber\\
&&\ \ \ \ \ -\frac{\partial U(\overline{A}(B),B)}{\partial B} [g_{\overline{A}(B),B}(U(A,B)-l)]
-\frac{\partial U(\overline{A}(B),B)}{\partial A} \frac{d\overline{A}(B)}{dB}[g_{\overline{A}(B),B}(U(A,B)-l)]\nonumber\\
&&\ \ \ \ \ +\int_{U(A,B)}^{u(A,B)} [-e^{-\lambda(x-a)}+\frac{d\overline{A}(B)}{dB}] dx\nonumber\\
&&\ \ \ =\int_{U(A,B)}^{u(A,B)} [-e^{-\lambda(x-a)}+e^{-\lambda(x_1(B)-a)}] dx\nonumber\\
&&\ \ \ >0,\nonumber
\end{eqnarray}
where the last inequality is due to $x_1(B)<U(A,B)<u(A,B)$ for $B\in(\underline{B}_1,\overline{B})$.
Thus, we have proved that
$\Lambda_2(\overline{A}(B), B)$ is strictly increasing in
$B\in (\underline{B}_1, \overline{B})$.

Because $\tilde g(\underline{B}_1)=k+\ell$, we have
\begin{displaymath}
\underline{A}(\underline{B}_1)=  g_{\underline{A}(\underline{B}_1),
  \underline{B}_1}(x_2(\underline{B}_1))-\ell =
  g_{\overline{A}(\underline{B}_1),
    \underline{B}_1}(x_1(\underline{B}_1))+k=\overline{A}(\underline{B}_1).
\end{displaymath}
Thus,
\begin{displaymath}
\lim_{B\downarrow \underline{B}_1}U(\overline{A}(B),B)=\lim_{B\downarrow
  \underline{B}_1}u(\overline{A}(B),B)=x_2(\underline{B}_1).
\end{displaymath}
It follows that
\begin{equation}
  \label{eq:Lambda1zero}
\lim_{B\downarrow\underline{B}_1}\Lambda_1(\overline{A}(B),B)=0.
\end{equation}
 We now show that
 \begin{equation}
   \label{eq:Lambda2barAinfinty}
\lim_{B\uparrow\overline{B}}\Lambda_1(\overline{A}(B),B)=\infty.
 \end{equation}
It is clear that (\ref{eq:Lambda1zero}),
(\ref{eq:Lambda2barAinfinty}) and the monotonicity imply the existence of a unique
$\underline{B}_2\in (\underline{B}_1, \overline{B})$ that satisfies
(\ref{eq:B2}).

To prove (\ref{eq:Lambda2barAinfinty}), one can check that
\begin{eqnarray}
\frac{d U(\overline{A}(B),B)}{d B}
&=&\frac{e^{-\lambda(U(\overline{A}(B),B)-a)}-e^{-\lambda(x_1(B))-a)}}{g_{\overline{A}(B),B}'(U(\overline{A}(B),B))}<0,\label{eq:dU/dA}\\
\frac{d u(\overline{A}(B),B)}{d B}
&=&\frac{e^{-\lambda(u(\overline{A}(B),B)-a)}-e^{-\lambda(x_1(B))-a)}}{g_{\overline{A}(B),B}'(u(\overline{A}(B)))}>0.\label{eq:du/dA}
\end{eqnarray}
Therefore, $U(\overline{A}(B), B)$ decreases in $B$ and $u(\overline{A}(B), B)$
increases in $B$. Thus,
\begin{displaymath}
  U(\overline{A}(B), B)\le U(\overline{A}(\underline{B}_1), \underline{B}_1)  \quad
  \text{and} \quad
  u(\overline{A}(B), B)\ge u(\overline{A}(\underline{B}_1), \underline{B}_1)
\end{displaymath}
as $B\uparrow \overline{B}$. Therefore,
noting that $g_{\overline{A}(B),B}(x)-l\geq0$ for $x\in(U(\overline{A}(B), B),u(\overline{A}(B), B))$, we have
\begin{eqnarray*}
\lim_{B\uparrow\overline{B}}\Lambda_2(\overline{A}(B),B)
&=&\lim_{B\uparrow\overline{B}}\int_{U(\overline{A}(B),B)}^{u(\overline{A}(B),B)}\Big[g_{\overline{A}(B),B}(x)-l\Big]dx
\\
&\ge&\lim_{B\uparrow\overline{B}}\int_{U(\overline{A}(\underline{B}_1),\underline{B}_1)}^{u(\overline{A}(\underline{B}_1),\underline{B}_1)}
\Big[g_{\overline{A}(B),B}(x)-l\Big]dx
\\
&=&\lim_{B\uparrow\overline{B}}\int_{U(\overline{A}(\underline{B}_1),\underline{B}_1)}^{u(\overline{A}(\underline{B}_1),\underline{B}_1)}\Big
[\overline{A}(B) + g_{0,B}(x)-l\Big]dx
\\
&=&\bigl
(U(\overline{A}(\underline{B}_1),\underline{B}_1) -
u(\overline{A}(\underline{B}_1),\underline{B}_1)\bigr)
\lim_{B\uparrow\overline{B}} \overline{A}(B) \\
&& { }+
\int_{U(\overline{A}(\underline{B}_1),\underline{B}_1)}^{u(\overline{A}(\underline{B}_1),\underline{B}_1)}\Big
[g_{0,\overline{B}}(x)-l\Big]dx
\\
&=& \infty,
\end{eqnarray*}
where we have used the fact that
\begin{displaymath}
  \lim_{B\uparrow\overline{B}} \overline{A}(B)=
  \lim_{B\uparrow\overline{B}} h_1(x_1(B))/\mu -k=\infty.
\end{displaymath}
\end{proof}

Lemma \ref{lem:Lambda2continuity} and the inequality in (\ref{eq:B2})
immediately imply the following
lemma
\begin{lemma}\label{lem:Astar}
For each $B\in [\underline{B}_2, \overline{B})$, there exists a unqiue
$ A^*(B) \in (\underline{A}(B), \overline{A}(B)]$
such that
\begin{equation}
  \label{eq:Astar}
\Lambda_2(A^*(B), B)=L.
\end{equation}

\end{lemma}

Finally, we prove the following lemma, which in turn proves Theorem
\ref{lem:optimalParameters}.
\begin{lemma}\label{lem:d D}
There exist unique $B^*\in (\underline{B}_2, \overline{B})$,
$d^*$ and $D^*$ that  satisfy
\begin{eqnarray*}
&&d^* <x_1(B^*)<D^*< U(A^*(B^*), B^*),\\
&&g_{A^*(B^*),B^*}(d^*)=g_{A^*(B^*),B^*}(D^*)=-k,\\
&&g_{A^*(B^*),B^*}'(d^*)<0,\ g_{A^*(B^*),B^*}'(D^*)>0,\\
&&\int_{d^*}^{D^*} [g_{A^*(B^*),B^*}(x)+k]dx=-K.
\end{eqnarray*}
\end{lemma}
\begin{proof}
By  Lemma~\ref{lem:Astar} and the inequality in (\ref{eq:B2}),
for $B\in(\underline{B}_2,\overline{B})$, we have
$A^*(B)\in (\underline{A}(B), \overline{A}(B))$.
It follows from  (\ref{eq:gless-k}) that
\begin{displaymath}
g_{A^*(B),B}(x_1(B))=A^*(B)-h(x_1(B))<-k.
\end{displaymath}
Therefore, there exist unique $d(B)$ and $D(B)$ such that
\begin{eqnarray*}
&&d(B)<x_1(B)<D(B)<U(A^*(B), B),\\
&&g_{A^*(B),B}(d(B))=g_{A^*(B),B}(D(B))=-k,\\
&&g_{A^*(B),B}'(d(B))<0,\ g_{A^*(B),B}'(D(B))>0.
\end{eqnarray*}
Let
\begin{equation}
  \label{eq:Lambda1Astar}
  \Lambda_1(A^*(B),B)=\int_{d(B)}^{D(B)}[g_{A^*(B),B}(x)+k]dx.
\end{equation}
We are going to prove that
$\Lambda_1(A^*(B),B)$ is continuous and strictly decreasing in
$B\in[\underline{B}_2,\overline{B})$
and
\begin{equation*}
  \lim_{B\downarrow \underline{B}_2} \Lambda_1(A^*(B),B)=0 \quad \text{and}
  \quad
  \lim_{B\uparrow \overline{B}} \Lambda_1(A^*(B),B)=-\infty.
\end{equation*}
Therefore,
there exists a unique $B^*\in(\underline{B}_2,\overline{B})$
such that 
\begin{eqnarray*}
\Lambda_1(A^*(B^*),B^*)=-K,
\end{eqnarray*}
from which one proves the lemma.

To prove that
$\Lambda_1(A^*(B),B)$ is continuous and strictly decreasing in
$B\in[\underline{B}_2,\overline{B})$,
we apply the {\em Implicit Function Theorem} to (\ref{eq:Astar}). We have
\begin{eqnarray}
\label{eq:dA/dB}
\frac{d A^*(B)}{d B}
=\frac{\int_{U(A^*(B),B)}^{u(A^*(B),B)}e^{-\lambda(x-a)}dx}{u(A^*(B),B)-U(A^*(B),B)}
>0.
\end{eqnarray}
Equation (\ref{eq:dA/dB}) yields that, for $x\in[d(B),D(B)]$
\begin{eqnarray}
\frac{\partial g_{A^*(B),B}(x)}{\partial B} \label{eq:dgAstardB}
&=&\frac{d A^*(B)}{d B}-e^{-\lambda (x-a)}\nonumber\\
&=&\frac{\int_{U(A^*(B),B)}^{u(A^*(B),B)}\Big[e^{-\lambda (y-a)}-e^{-\lambda(x-a)}\Big]dy}{u(A^*(B),B)-U(A^*(B),B)} \nonumber\\
&<&0. \nonumber
\end{eqnarray}
This in turn  implies that
\begin{equation}
\label{eq:dLambda1AstardB}
\frac{\partial \Lambda_1(A^*(B),B)}{\partial B}=\int_{d(B)}^{D(B)}
\frac{\partial g_{A^*(B),B}(x)}{\partial B}dx<0.
 \end{equation}
Therefore,  $\Lambda_1(A^*(B),B)$ is strictly decreasing in  $B\in[\underline{B}_2,\overline{B})$.

It follows from
\eqref{eq:B2} and Lemma \ref{lem:Astar} that
\begin{eqnarray*}
A^*(\underline{B}_2)=\overline{A}(\underline{B}_2).
\end{eqnarray*}
This, together with the definition of $\overline{A}(B)$ in
\eqref{eq:AoverlineB},
shows that
\begin{equation}
  \label{eq:gB2minusk}
g_{A^*(\underline{B}_2),\underline{B}_2}(x_1(\underline{B}_2))
=g_{\overline{A}(\underline{B}_2),\underline{B}_2}(x_1(\underline{B}_2))
=\overline{A}(\underline{B}_2)-\frac{1}{\mu}h(x_1(\underline{B}_2))
=-k.
\end{equation}
Therefore, we have
\begin{displaymath}
\lim_{B\downarrow \underline{B}_2}D(B)=\lim_{B\downarrow
  \underline{B}_2}d(B)=x_1(\underline{B}_2).
\end{displaymath}
It follows that
\begin{equation}
  \label{eq:Lambda1AstarupperLimit}
\lim_{B\downarrow \underline{B}_2}\Lambda_1(A^*(B),B)=0.
\end{equation}
It remains to prove
\begin{equation}
  \label{eq:Lambda1negativeInfinity}
  \lim_{B\uparrow
  \overline{B}}\Lambda_1(A^*(B),B)=-\infty.
\end{equation}
For $B\in(\underline{B}_2,\overline{B})$,
\begin{eqnarray*}
\frac{\partial g_{A^*(B),B}(x_1(B))}{\partial B}
&=&\frac{d A^*(B)}{d B}-e^{-\lambda (x_1(B)-a)}+g_{A^*(B),B}'(x_1(B))\frac{d x_1(B)}{d B}\\
&=&\frac{\int_{U(A^*(B),B)}^{u(A^*(B),B)}\Big[e^{-\lambda(y-a)}-e^{-\lambda (x_1(B)-a)}\Big]dy}{u(A^*(B),B)-U(A^*(B),B)} \nonumber\\
&<&0,
\end{eqnarray*}
which, together with (\ref{eq:gB2minusk}), implies that
\begin{equation}
  \label{eq:glessminusk}
g_{A^*(B),B}(x_1(B))<-k
\end{equation}
for each $B\in(\underline{B}_2,\overline{B})$. Fix  a
$\underline{B}_3\in(\underline{B}_2,\overline{B})$ and  let
\begin{displaymath}
  M_1 = \Bigl(-k - g_{A^*(\underline{B}_3),\underline{B}_3}(x_1(\underline{B}_3))\Bigr)/2.
\end{displaymath}
It follows from (\ref{eq:glessminusk}) that $M_1>0$.
Then for each $B\in(\underline{B}_3, \overline{B})$,
\begin{displaymath}
  g_{A^*(B),B}(x_1(B))<   g_{A^*(\underline{B}_3),\underline{B}_3}(x_1(\underline{B}_3))
  = -k -2M_1<-k-M_1.
\end{displaymath}
Therefore,  for each $B\in(\underline{B}_3, \overline{B})$ there exist unique
$d_1(B)$ and $D_1(B)$ such that
\begin{eqnarray}
&&d_1(B)<x_1(B)<D_1(B),\nonumber\\
&&g_{A^*(B),B}(d_1(B))=g_{A^*(B),B}(D_1(B))=-k-M_1,\label{eq:d1D1}\\
&&g_{A^*(B),B}'(d_1(B))<0,\ g_{A^*(B),B}'(D_1(B))>0. \nonumber
\end{eqnarray}
The properties of $g$ in Lemma~\ref{thm:optimalParameter} (see also
Figure \ref{fig:g1})
imply that  for each $B\in(\underline{B}_3, \overline{B})$
\begin{eqnarray*}
d(B)<d_1(B)<x_1(B)<D_1(B)<D(B).
\end{eqnarray*}
This, together with (\ref{eq:x2Blimit})  implies that
\begin{eqnarray}
\label{eq:lim d_1(B)}
\lim_{B\uparrow\overline{B}} d_1(B)=-\infty.
\end{eqnarray}
Note that for $x\in (d(B),D(B))$, $g_{A^*(B),B}(x)<-k$.
Therefore, for $B\in(\underline{B}_3, \overline{B})$,
\begin{eqnarray*}
\Lambda_1(A^*(B),B)
&=&\int_{d(B)}^{D(B)}[g_{A^*(B),B}(x)+k]dx\\
&\leq&\int_{d_1(B)}^{D_1(B)}[g_{A^*(B),B}(x)+k]dx\\
&\leq&\int_{d_1(B)}^{D_1(B)}[-M_1]dx\\
&=&-M_1(D_1(B)-d_1(B)).
\end{eqnarray*}
It follows from (\ref{eq:d1D1}) and (\ref{eq:dA/dB}) that
 for each $B\in(\underline{B}_3, \overline{B})$,
\begin{eqnarray*}
\frac{d D_1(B)}{d B}
=\frac{\int_{U(A^*(B),B)}^{u(A^*(B),B)}\Big[e^{-\lambda (D_1(B)-a)}-e^{-\lambda(x-a)}\Big]dx}{(u(A^*(B),B)-U(A^*(B),B))g'(D_1(B))}
>0.
\end{eqnarray*}
Thus, for
any $B\in(\underline{B}_3, \overline{B})$,
\begin{displaymath}
  D_1(B) \ge D_1(\underline{B}_3).
\end{displaymath}
Thus, for any $B\in(\underline{B}_3, \overline{B})$,
\begin{equation}
  \label{eq:Lambda1Bound}
  \Lambda_1(A^*(B), B) \le M_1 d_1(B) - M_1 D_1(\underline{B}_3).
\end{equation}
Now (\ref{eq:Lambda1negativeInfinity}) readily follows from
(\ref{eq:Lambda1Bound}) and  (\ref{eq:lim d_1(B)}).
\end{proof}

\section{Singular  Controls}
\label{sec:instanteneous}

In this section, we assume that  $K=0$ and $L=0$. Therefore, we
restrict our feasible policies to singular controls also known as
instantaneous controls  as in
(\ref{eq:Y1N1xi1}) and (\ref{eq:Y2N2xi2}).
A two-parameter control band policy is defined by two
parameters $d$, $u$, where $d<u$. No control is
exercised until the inventory level $Z(t)$ reaches the lower boundary $d$ or
the upper boundary $u$. When $Z(t)$ reaches a boundary, there is no
advantage in using impulse control because there is no fixed cost.

\subsection{Control Band Policies}
\label{sec:controlLimit}
Let us fix a two-parameter control band policy $\varphi=\{d, u\}$.
To mathematically describe the control process $(Y_1, Y_2)$, we need to
use two-sided regulator: for each $x\in \D$ with $x(0)\in [d, u]$,
find a triple $(y_1, y_2, z)\in \D^3$
such that
\begin{eqnarray}
 &&  z(t) = x(t) + y_1(t) - y_2(t), \quad t\ge 0, \label{eq:2sided-1}\\
 &&  z(t) \in [d, u], \quad t\ge 0, \\
 &&  y_1(0)=y_2(0)=0,\quad   y_1 \text{ and } y_2 \text{ are nondecreasing}, \\
 && \text{$y_1$ and $y_2$ increases only when $z=d$ and $z=u$, respectively.} \label{eq:2sided-4}
\end{eqnarray}
The precise mathematical meaning of (\ref{eq:2sided-4}) is
\begin{equation}
  \label{eq:2-sided-4b}
  \int_0^\infty (z(t)-d)\, dy_1(t) =0 \quad \text{and}\quad
  \int_0^\infty (u-z(t))\, dy_2(t) =0.
\end{equation}
One can verify that (\ref{eq:2-sided-4b}) is equivalent to the
following: whenever $z(t)>d$ for $t\in [t_1, t_2]$,
$y_1(t_2)-y_1(t_1)=0$ and whenever $z(t)<u$ for $t\in [t_1, t_2]$,
$y_2(t_2)-y_2(t_1)=0$. Lemma \ref{lem:2sidedRegulator} below follows
from Proposition 6 in Section 2.4 of \cite{har85}. That proposition is
stated for each continuous path $x\in \D$; one can verify
that the proposition continues to hold when the continuity of $x$ is
dropped.

\begin{lemma}
\label{lem:2sidedRegulator}
  For each $x\in \D$ with $x(0)\in [d, u]$, there exists a unique
  triple $(y_1, y_2, z)\in \D^3$ that satisfies
  (\ref{eq:2sided-1})-(\ref{eq:2-sided-4b}).
\end{lemma}
The lemma asserts that the map $\Psi: x\in \D_0\to (y_1, y_2, z)\in
\D^3$ is well defined, where $\D_0= \{x\in \D: x(0)\in[d, u]\}$.
 In the following, we use notation
\begin{displaymath}
  y_1= \Psi_1(x), \quad
  y_2= \Psi_2(x), \quad \text{and }
  z = \Psi_3(x).
\end{displaymath}
The nondecreasing functions $(y_1, y_2)$ are said to be the two-sided
regulator of $x$, and $z$ is the regulated path of $x$.
When either $u=\infty$ or $d=-\infty$, the corresponding one-sided
regular is defined in Section 2.2 of \cite{har85}.

Under the control band policy $\{d, u\}$ with initial inventory level
$x\in[d, u]$, the controls $(Y_1, Y_2)$ are given
by $Y_1=\Psi_1(X)$,  $Y_2=\Psi_2(X)$, and the inventory process
$Z=\Psi_3(X)$.

To find the long-run average cost under the policy $\varphi=\{d, u\}$, we use
the following theorem.
\begin{theorem}
\label{thm:controllimit}
Fix a control band policy $\varphi=\{d, u\}$. If there
exist a constant $\gamma$ and a twice continuously differentiable function
$V:[d,u]\rightarrow \R$ that satisfies
\begin{equation}
 \Gamma V(x)+h(x)=\gamma, \quad d\leq x\leq u,\label{eq:Poissonab}
\end{equation}
with boundary conditions
\begin{eqnarray}
 && V'(d)=-k,\label{eq:Va}\\
 && V'(u)= \ell .\label{eq:Vb}
\end{eqnarray}
Then the average cost $AC(x,\varphi)$ is independent of the initial
inventory level $x\in \R$ and is given by $\gamma$ in \eqref{eq:Poissonab}.
\end{theorem}
\begin{proof}
First we assume $x\in [d, u]$. In this case, $Z(0)=x$.
  By It\^{o}'s formula,
    \begin{eqnarray*}
    V(Z(t)) & = & V(Z(0)) + \int_0^t \Gamma V(Z(s))ds + \sigma \int_0^t V'(Z(s))
    dW(s) + \int_0^t V'(Z(s))dY_1(s)\\
   && { }-\int_0^t V'(Z(s))dY_2(s) \\
 &=&   V(Z(0)) + \int_0^t \Gamma V(Z(s))ds + \sigma \int_0^t V'(Z(s))
    dW(s) +  V'(d)Y_1(t)- V'(u)Y_2(t) \\
 &=&   V(Z(0)) + \gamma t - \int_0^t h(Z(s))ds + \sigma \int_0^t V'(Z(s))
    dW(s) -k Y_1(t)- \ell Y_2(t).
    \end{eqnarray*}
Therefore
\begin{displaymath}
  \E_x[V(Z(t))] = \E_x[V(Z(0))] + \gamma t - \biggl (\int_0^t h(Z(s))ds
 + k Y_1(t) + \ell Y_2(t)  \biggr ).
\end{displaymath}
Dividing both sides by $t$ and taking the limit as $t\to\infty$,
we have $\AC(x,\varphi)=\gamma$.

When $x\not\in [d, u]$, we assume $Z$ immediately jumps to the closest
point in $[d, u]$ at time $0$. Therefore, $Z(0)=d$ if $x<d$ and
$Z(0)=u$ if $x>u$. Since $Z(0)\in [d, u]$, the rest of the proof is
identical to the case when $x\in [d, u]$.

\end{proof}

\begin{proposition}
 Let $\varphi=\{d,u\}$ be a control band policy with
  \begin{displaymath}
    d<u.
  \end{displaymath}
Let  $m \in \R$ be any fixed number. Define
\begin{displaymath}
  V(x)=  \int_{m}^x g(y) dy
\end{displaymath}
with
\begin{displaymath}
  g(x)= V'(m) e^{\lambda (m-x)}  + \gamma \frac{2}{\sigma^2} \int_m^x
e^{\lambda (y-x)}dy -\frac{2}{\sigma^2}\int_m^x h(y)e^{\lambda (y-x)}dy,
\end{displaymath}
where
\begin{eqnarray}
 &&  \gamma = \frac{d_1(f_2+\ell)+d_2(f_1+k)}{d_1e_2+d_2e_1}, \label{eq:gamma-instantaneous} \\
 && V'(m) = \frac{e_1(f_2+\ell)+e_2(f_1+k)}{d_1e_2+d_2e_1}. \label{eq:Vprime-instantaneous}
\end{eqnarray}
Then $(V, \gamma)$ is a solution to
(\ref{eq:Poissonab})-(\ref{eq:Vb}). In (\ref{eq:gamma-instantaneous}) and
(\ref{eq:Vprime-instantaneous}), we set
\begin{eqnarray}
&&d_1=e^{\lambda (m-d)}, \quad  d_2=e^{\lambda (m-u)}, \label{eq:coeffd}\\
&&e_1=-\frac{2}{\sigma^2} \int_m^d e^{\lambda (y-d)}dy, \quad
e_2=\frac{2}{\sigma^2} \int_m^u e^{\lambda (y-u)}dy, \label{eq:coeffe}\\
&&f_1=-\frac{2}{\sigma^2}\int_m^d h(y)e^{\lambda (y-d)}dy , \quad
f_2=\frac{2}{\sigma^2}\int_m^u h(y)e^{\lambda (y-u)}dy. \label{eq:coefff}
\end{eqnarray}
\end{proposition}
\begin{proof}
Similar to the proof of Proposition \ref{prop:controlband},
equation (\ref{eq:Poissonab}) implies that
\begin{displaymath}
  V'(x)=
e^{\lambda (m-x)}  V'(m) + \gamma \frac{2}{\sigma^2} \int_m^x
e^{\lambda (y-x)}dy -\frac{2}{\sigma^2}\int_m^x h(y)e^{\lambda (y-x)}dy.
\end{displaymath}
Boundary conditions (\ref{eq:Va}) and (\ref{eq:Va}) become
\begin{eqnarray}
e^{\lambda (m-d)}  V'(m) + \gamma \frac{2}{\sigma^2} \int_m^d
e^{\lambda (y-d)}dy -\frac{2}{\sigma^2}\int_m^d h(y)e^{\lambda (y-d)}dy=-k,\label{eq:boundaryk}\\
e^{\lambda (m-u)}  V'(m) + \gamma \frac{2}{\sigma^2} \int_m^u
e^{\lambda (y-u)}dy -\frac{2}{\sigma^2}\int_m^u h(y)e^{\lambda (y-u)}dy=\ell.\label{eq:boundaryl}
\end{eqnarray}
Using the coefficients defined in (\ref{eq:coeffd})-(\ref{eq:coefff}),
we see the
 boundary conditions (\ref{eq:boundaryk}) and
(\ref{eq:boundaryl}) become
\begin{eqnarray*}
&&  d_1 V'(m) -\gamma e_1 = -(k+f_1), \\
&&  d_2 V'(m) + \gamma e_2 = \ell+f_2,
\end{eqnarray*}
from which we have unique solution for $\gamma$ and  $V'(m)$
given in (\ref{eq:gamma-instantaneous}) and (\ref{eq:Vprime-instantaneous}).
\end{proof}

\subsection{Optimal Policy and Optimal Parameters}
\label{sec:optimalinstantaneous}

Theorem \ref{thm:lowerbound} suggests the following strategy to obtain
an optimal policy. We hope the optimal policy is a control band
policy. Therefore, the first task is to find an optimal control band
policy among all control band policies. Denote this optimal control
band policy by $\varphi^*=\{d^*, u^*\}$, $d^*<u^*$,  with long-run average cost
$\gamma^*$.  We hope that $\gamma^*$ can be used  the constant in
\eqref{eq:lbPoission} of Theorem \ref{thm:lowerbound}. To find the
corresponding $f$ that satisfies all the conditions of  Theorem
\ref{thm:lowerbound}, we start with the relative value function
$V(x)$ associated with the policy $\varphi^*$ that is defined on the
interval $[d^*, u^*]$. We need to extend $V(x)$ so that it can be
defined on $\R$. Given that $V(x)$ is the relative value function,
it is natural to define
\begin{eqnarray}
\label{eq:f-instantaneous}
  f(x) =
  \begin{cases}
   V(d^*)+k(d^*-x)  & \text{ for } x<d^*, \\
   V(x) & \text{ for } d^*\leq x\leq u^*, \\
   V(u^*)+\ell(x-u^*)  & \text{ for } x>u^*.
  \end{cases}
\end{eqnarray}
Since we wish $f\in C^1(\R)$, we should have
\begin{equation}
  \label{eq:Boundarykl-instantaneous}
  V'(d^*)=-k, \quad  V'(u^*)=\ell.
\end{equation}
We also hope $f\in C^2(\R)$, we should have
the following conditions,
\begin{eqnarray}
\label{eq:b1}
 && V''(d^*) = 0, \quad V''(u^*) = 0.\label{eq:BoundaryV''-instantaneous}
\end{eqnarray}

In this section, we will first prove the existence of parameters
$d^*$ and $u^*$ such that the relative value function $V$ corresponding
the control band policy $\varphi=\{d^*, u^*\}$ satisfies
(\ref{eq:Poissonab})-(\ref{eq:Vb}), and
(\ref{eq:Boundarykl-instantaneous})-(\ref{eq:BoundaryV''-instantaneous}). Since part of the solution
is to find the boundary points $d^*$ and $u^*$,
equations (\ref{eq:Poissonab})-(\ref{eq:Vb}), and
(\ref{eq:Boundarykl-instantaneous})-(\ref{eq:BoundaryV''-instantaneous}) define a free boundary
problem.
We then prove that the extension $f$ in (\ref{eq:f-instantaneous}) and
$\gamma^*=\AC(\varphi^*, x)$ jointly satisfy all the conditions in
Theorem  \ref{thm:lowerbound}.

In the rest of this section, we assume that $\mu>0$.
The statement and analysis for the cases $\mu<0$ and $\mu=0$ are
analogous and are omitted.
Recall the function $g(x)=g_{A, B}(x)$ defined in (\ref{eq:g1}).

\begin{figure}[tb]
  \begin{center}
   \includegraphics[width=10cm]{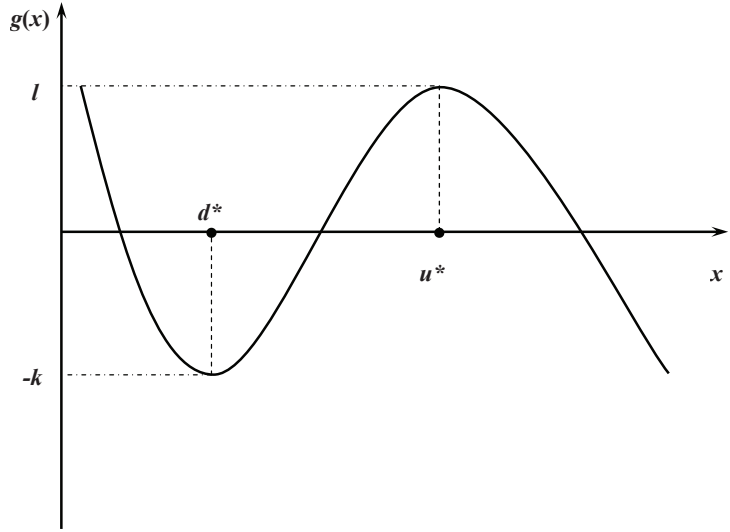}\label{fig:instan:g1}
\caption{There exist unique $d^*=x_1(\underline{B}_1)$ and
  $u^*=x_2(\underline{B}_1)$.}
\end{center}
\end{figure}

\begin{theorem}
\label{lem:optimalParameter-instantaneous}
There exist unique $A^*$, $B^*$, $d^*$ and $u^*$
such that $g(x)=g_{A^*,B^*}(x)$, $d^*$ and $u^*$ satisfy
\begin{eqnarray}
 &&  g(d^*) = -k, \label{eq:Boundary-g1}\\
 &&  g(u^*) = \ell , \label{eq:Boundary-g2}\\
 && g'(d^*) = 0, \label{eq:Boundary-g3}\\
 && g'(u^*) = 0.\label{eq:Boundary-g4}
\end{eqnarray}
Furthermore, $g(x)$ decreases in $(-\infty, d^*)$, increases in $(d^*,
u^*)$, and decreases again in $(u^*, \infty)$.
\end{theorem}

\begin{proof}
Recall  the definition of $\overline{B}$ in
(\ref{eq:under b over B}).
For each $B\in (0, \overline{B})$, by Lemma \ref{lem:x_i},
there is a unique local minimizer $x_1(B)<a$ and a unique local
maximizer $x_2(B)>a$ for function $g_{0, B}(x)$. By Lemma \ref{lem:B1},  there
exists a unique $\underline{B}_1\in (0, \overline{B})$ that satisfies
(\ref{eq:B1}). Let
\begin{displaymath}
A^*=h(x_1(\underline{B}_1))/\mu + \ell = h(x_2(\underline{B}_1))/\mu-k.
\end{displaymath}
Then $g(x)=g_{A^*, \underline{B}_1}(x)$, $d^*=x_1(\underline{B}_1)$ and
$u^*=x_2(\underline{B}_1)$ satisfy
(\ref{eq:Boundary-g1})-(\ref{eq:Boundary-g4});
see Figure~\ref{fig:instan:g1}.
\end{proof}

Now we show that the control band policy $\varphi^*=\{d^*,u^*\}$
is optimal policy among all feasible policies.

\begin{theorem}
\label{thm:optimal-instantaneous}
  Assume that $h$ satisfies Assumption
  \ref{assumption:h}. Let $d^*$ and $u^*$, along with constants $A^*$ and $B^*$,
  be the unique solution in
  Theorem~\ref{lem:optimalParameter-instantaneous}. Then the control
  band policy
  $\varphi^*=\{d^*, u^*\}$ is optimal among all feasible policies.
\end{theorem}

\begin{proof}
Let $g(x)$, $x\in \R$, be the function in (\ref{eq:g1}) with $A=A^*$
and $B=B^*$. Let
\begin{displaymath}
  \bar g(x) =
  \begin{cases}
    -k, & x< d^*, \\
    g(x), & d^*\le x \le u^*, \\
    \ell, &  x>u^*.
  \end{cases}
\end{displaymath}
Define
\begin{eqnarray}
\label{eq:V-instantaneous}
  V(x) = \int_{d^*}^x \bar g(y)dy.
\end{eqnarray}
Let $\gamma^*$ be the long-run average cost under policy
$\varphi^*$. We now show that $V$ and $\gamma^*$ satisfy all the
conditions in Theorem \ref{thm:lowerbound}. Thus, Theorem
\ref{thm:lowerbound} shows that the long-run average cost under any
policy is at least $\gamma^*$. Therefore, $\gamma^*$ is the optimal
cost and the control band policy $\varphi^*$ is an optimal policy.
Now we check that $V(x)$ is in $C^2(\R)$ and satisfies (\ref{eq:lbPoission})-(\ref{eq:lbL}).

First,  $V(x)$ is in $C^2([d^*, u^*])$.
Lemma~\ref{lem:optimalParameter-instantaneous} and the definition of $V$ in (\ref{eq:V-instantaneous}) imply that
\begin{eqnarray*}
\lim_{x\uparrow d^*} V''(x)=0=\lim_{x\downarrow d^*} V''(x),
\mbox{ and } \lim_{x\uparrow u^*} V''(x)=0=\lim_{x\downarrow u^*} V''(x).
\end{eqnarray*}
Then, $V''(x)$ is continuous at $d^*$ and $u^*$.
Note that $V''(x)=0$ in $(-\infty,d^*)$ and $(u^*,+\infty)$.
Therefore, $V(x)$ is in $C^2(\R)$.
Let
\begin{eqnarray*}
M=\sup_{x\in[d^*,u^*]}\abs{g(x)},
\end{eqnarray*}
we have $\abs{V'(x)}\leq M$ for all $x\in \R$.

To check (\ref{eq:lbPoission}), we first find that
$\Gamma V(x)+h(x)=\gamma^*$ for $d^*\leq x\leq u^*$.
For $x<d^*$,
\begin{eqnarray*}
&&\Gamma V(x)+h(x)\\
&=& \frac{\sigma^2}{2}V''(x)+\mu V'(x)+h(x)\\
&=& \frac{\sigma^2}{2}V''(d^*)+\mu V'(d^*)+h(x)\\
&\geq& \frac{\sigma^2}{2}V''(d^*)+\mu V'(d^*)+h(d^*)\\
&=&\gamma^*,
\end{eqnarray*}
where the second equality is because for $x<d^*$,
$V''(x)=0=V''(d^*)$ and $V'(x)=-k=V'(d^*)$,
the inequality is due to $x<d^*=x_1<a$, where $a$ again is the
minimum point of $h$.
Similarly, for $x>u^*$, $\Gamma V(x)+h(x)\geq\gamma^*$.

Finally, (\ref{eq:lbK}) and (\ref{eq:lbL}) hold because
$\overline{g}(x)$ is strictly increasing in $x$, $x\in[d^*,u^*]$, and
$\overline{g}(d^*)=g(d^*)=-k$, $\overline{g}(u^*)=g(u^*)=\ell$ (See
Figure \ref{fig:instan:g1}). Thus, the optimality of control band
policy $\varphi^*$ is implied by Theorem \ref{thm:lowerbound}.

\end{proof}

\section{No Inventory Backlog }
\label{sec:nonnegative}
In this section, the inventory backlog is not allowed and thus we add
the constraint $Z(t)\geq 0$ for all $t\geq 0$. The holding cost
function $h(\cdot)$ is defined on $[0,\infty)$, and $a\in [0, \infty)$
is its minimum point. We focus on the impulse control case when $K>0$ and
$L>0$. Thus, this section parallels
Section~\ref{sec:impulse-controls}. In particular, the results and
proofs in
this section are analogous to that in Section
\ref{sec:impulse-controls}. In our presentation, we will highlight the
differences.

For a control band policy $\{d,D,U,u\}$ with $0\leq d<D<U<u$,
one can continue to use Theorem \ref{thm:control band} to evaluate its
performance and to obtain is relative value function.
But the lower bound theorem, Theorem~\ref{thm:lowerbound}, needs to be
slightly modified as in the following theorem.
\begin{theorem}
\label{thm:lowerbound nonnegative}
Suppose that $f\in C^1([0,+\infty))$ and $f'$ is absolutely continuous such
  that $f''$ is
  locally $L^1$. Suppose that there exists a constant $M>0$ such
  that  $|f'(x)|\leq M $ for all $x\in [0,+\infty)$.
  Assume further that
  \begin{eqnarray}
    && \Gamma f(x)+h(x)\geq \gamma \mbox{ for $x\in
      [0,+\infty)$},\label{eq:lbPoission nonnegative}\\
    && f(y) - f(x)\le K+ k(x-y) \mbox{ for $0\leq y<x$},\label{eq:lbK nonnegative} \\
    && f(y) - f(x)\le L+\ell(y-x) \mbox{ for $0\leq x<y$}.\label{eq:lbL nonnegative}
  \end{eqnarray}
Then $\AC(x,\varphi)\geq \gamma$ for each feasible policy $\varphi$  and each
initial state $x\in [0,+\infty)$.
\end{theorem}

\subsection{Optimal Policy Parameters}

Recall that for a given set of parameters $\{d,D,U,u\}$ with
$0\le d<D<U<u$, the
corresponding relative value
function satisfies (\ref{Poisson equ})-(\ref{V(u)-V(U)}).
To search for the optimal parameters $(d^*, D^*, U^*, u^*)$, we impose
the following
conditions on $\{d,D,U,u\}$
and $V$:
\begin{eqnarray}
 && V'(U)=l, \label{eq:V'(U) nonnegative}\\
 && V'(u)=l, \label{eq:V'(u) nonnegative}\\
 && V'(D)=-k, \label{eq:V'(D) nonnegative}\\
 && V'(d)=-k-\alpha, \label{eq:V'(d) nonnegative}\\
 && 0\le d< D< U<u, \label{eq:d>0 nonnegative} \\
 && \alpha d=0,\ \ \mbox{and} \label{eq:Lagarange nonnegative}\\
 && \alpha\ge 0.
\end{eqnarray}
In some cases, it is optimal to have $d^*=0$. In such a case,
one  only needs to solve for three parameters $D^*$, $U^*$ and $u^*$.
\begin{figure}[t]
  \begin{center}
\subfigure[Case $d>0$]
   {\includegraphics[width=7cm]{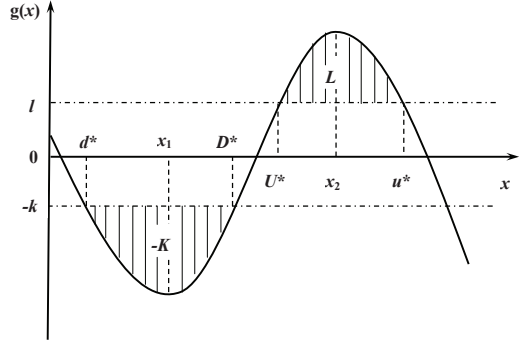}\label{fig:g2a}}\\
\subfigure[Case $d=0$]
  {\includegraphics[width=7cm]{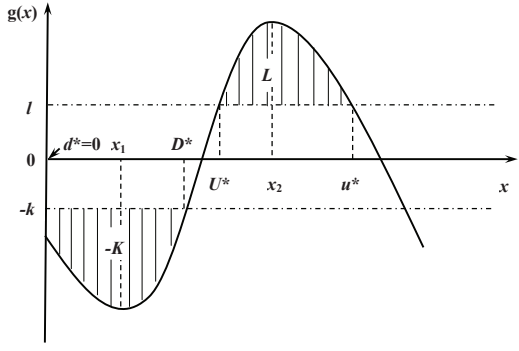}\label{fig:g2b}}
\subfigure[Case $d=0$]
  {\includegraphics[width=7cm]{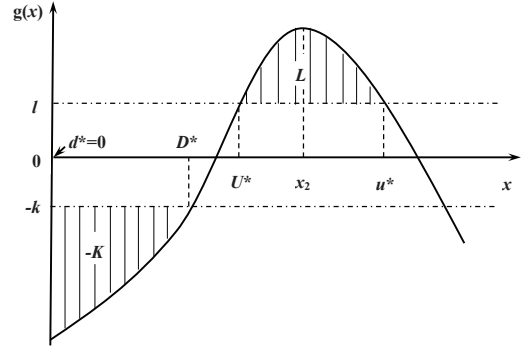}\label{fig:g2c}}
\caption{In the nonnegative case, the optimal control band policy has two possible cases: $d>0$ or $d=0$.}
\end{center}
\end{figure}
This section is analogous to Section \ref{sec:optimal}. We highlight
the differences between these two sections and omit some details
to avoid repetition.

Recall that $a$ is the minimum point of the holding cost function
$h(x)$ on $[0,\infty)$. It is possible $a=0$ or $a>0$. In the
following, whenever Assumption
\ref{assumption:h} is invoked for $h$, any condition on $h(x)$ with
$x<0$ is ignored.
Similar to Lemma \ref{lem:solutionToPoisson}, we have the following
lemma.
\begin{lemma}\label{lem:solutionToPoissonPos}
    For each $A, B\in \R$, function $g(x)=g_{A, B}(x)$ in (\ref{eq:g1}) is a solution to
  equation
  \begin{equation}
    \Gamma g(x) + h'(x) = 0 \quad \text{ for all } x\in
    [0,\infty)\setminus\{a\},
  \label{eq:gPoisson nonnegative}
  \end{equation}
\end{lemma}
The following theorem solves the free boundary problem
when inventory backlog is not allowed.
\begin{theorem}
\label{lem:optimalParameter nonnegative}
Assume  that the holding cost
function $h$ satisfies conditions (a)-(d) of
Assumption \ref{assumption:h}.
There exists unique $A$, $B$, $d$, $D$, $U$ and $u$
with
  \begin{displaymath}
  0\leq d< D \quad \text{ and } \quad U<x_2<u
  \end{displaymath}
and
such that the corresponding $g(x)=g_{A,B}(x)$ satisfies
\begin{eqnarray}
  && \int_d^D [g(x)+k] dx = -K, \label{eq:gK nonnegative}\\
  && \int_U^d[g(x)-\ell] dx = L, \label{eq:gL nonnegative}\\
  && g(d)=-k-\alpha,\ \  g(D) =-k, \label{eq:gk nonnegative}\\
  && g(U)=g(u) =\ell,\label{eq:gell nonnegative}\\
  && \alpha d=0,\ \ \mbox{and} \label{eq:Lagarange nonnegative2}\\
  && \alpha \ge 0.
\end{eqnarray}
Furthermore, $g$ has a local minimum at $x_1\le a$ and $g$ has the maximum at
$x_2>a$. The function $g$ is decreasing on $(0, x_1)$, increasing on
$(x_1,x_2)$ and decreasing again on $(x_2, \infty)$.
\end{theorem}
We leave the proof of Theorem~\ref{lem:optimalParameter nonnegative}
to the end of this section.

\begin{theorem}
\label{thm:optimal nonnegative}
  Assume that the holding cost function $h$ satisfies conditions
  (a)-(d) of  Assumption
  \ref{assumption:h}. Let $0\le d^*<D^*<U^*<u^*$, along with constants
  $A^*$ and
  $B^*$, be the unique solution in Theorem
  \ref{lem:optimalParameter nonnegative}. Then the control band policy
  $\varphi^*=\{d^*, D^*, U^*, u^*\}$ is optimal among all
  feasible policies to minimize the long-run average cost when
  inventory backlog is not allowed.
\end{theorem}
\begin{proof}
The proof is identical to that of Theorem \ref{thm:optimal}.
\end{proof}

The rest of this section is devoted to the proof for
Theorem~\ref{lem:optimalParameter nonnegative}.
This proof  is similar to the proof of Theorem
\ref{lem:optimalParameters}. We provide an outline of the proof
for Theorem~\ref{lem:optimalParameter nonnegative}, highlighting
differences between the two proofs.
We only consider the case when
$\mu>0$. Other cases are analogous and are omitted.
Define
\begin{eqnarray}
\label{eq:under b over B nonnegative}
\overline{B}_1=-\frac{1}{\mu}\int_{0}^a
h'(y)e^{\lambda(y-a)}dy
\end{eqnarray}
Because $h'(x)<0$ for $x\in (0,a)$, $\overline{B}_1\ge  0$.

The following lemma is analogs to Lemma~\ref{thm:optimalParameter}. The only
difference is that the expression for  $x_1=x_1(B)$ has two forms in
Lemma \ref{thm:optimalParameterPos}.
\begin{lemma}
  \label{thm:optimalParameterPos}
  (a) For any $A\in \R$ and for each fixed
  $B\in(0,\infty)$, $g_{A, B}$ attains a unique
  minimum in $[0, a]$ at $x_1=x_1(B)\in [0, a]$. The
  function $g_{A, B}$ attains a unique maximum in $(a, \infty)$ at
  $x_2=x_2(B) \in (a, \infty)$. Both $x_1(B)$ and $x_2(B)$ are
  independent of $A$.

  (b) For each fixed $B\in(0,\infty)$, the local maximizer
  $x_2=x_2(B)$ is the unqiue solution in $(a, \infty)$ to
  (\ref{eq:xiBequation}). For $B\in(0,\overline{B}_1)$,
  the local minimizer $x_1=x_1(B)$ is the
  unique solution in $(0, a)$ to   (\ref{eq:xiBequation}).
  For $B\in [\overline{B}_1,\infty)$, $x_1=x_1(B)=0$.

  (c) For each $B\in(0,\infty)$, $g_{A, B}'(x)<0$
  for $x\in (0, x_1(B))$, $g_{A, B}'(x)>0$ for $x\in (x_1(B), x_2(B))$, and
  $g_{A, B}'(x)<0$ for $x\in(x_2(B), \infty)$.
\end{lemma}
The following lemma is analogs to Lemma~\ref{lem:x_i}.
\begin{lemma}
\label{lem:x_iPos}
(a) The local minimizer  $x_1(B)$  is continuous and nonincreasing
 in $B\in(0,\infty)$.
 The local maximizer $x_2(B)$  is continuous and strictly
increasing  in $B\in(0,\infty)$.
Furthermore, (\ref{eq:x1Blimit}) holds and
\begin{equation}
  \label{eq:x2BlimitPos}
 \lim_{B\uparrow \infty}  x_1(B) = 0 \quad \text{ and } \quad
 \lim_{B\uparrow \infty}  x_2(B)=\infty.
\end{equation}
(b) For each $B\in(0,\infty)$,
\begin{equation}
  \label{eq:ghPos}
  g_{A,B}(x_2(B))= A - h(x_2(B))/\mu.
\end{equation}
 For each $B\in(0,\overline{B}_1)$,
\begin{equation}
  \label{eq:ghx1Pos}
g_{A,B}(x_1(B))= A - h(x_1(B))/\mu.
\end{equation}
 For each $B\in[\overline{B}_1,\infty)$,
\begin{equation}
  \label{eq:ghx1Pos2}
  g_{A,B}(x_1(B))= g_{A,B}(0)= A- Be^{\lambda a}+\frac{\lambda}{\mu}
  \int_0^a h(-y+a)e^{-\lambda(y-a)}dy.
\end{equation}
\end{lemma}
\begin{proof}
(a) Note that $x_1(B)=0$ for $B\in[\overline{B}_1,\infty)$. Thus,
$x_1(B)$ is continuous for $B\in(\overline{B}_1,\infty)$.
It follows the proof of Lemma~\ref{lem:x_i}  that $x_1(B)$ is
continuously differentiable in $B\in(0,\overline{B}_1)$,
and $x_2(B)$ is continuously differential in $B\in(0,\infty)$.
One can easily check that  $x_1(B)$ is continuous at
$B=\overline{B}_1$ and that $x_i(B)$ has the desired monotonicity
property for $i=1, 2$. Limits in (\ref{eq:x1Blimit}) can be obtained similarly as
in Lemma~\ref{lem:x_i}.  The limit in the left side of
(\ref{eq:x2BlimitPos}) follows from $x_1(B)=0$ for $B\in(\overline{B}_1,\infty)$.
The limit in the right side of (\ref{eq:x2BlimitPos}) follows from equation (\ref{eq:xiBequation})
for the definition of $x_2(B)$.

(b) Equations \eqref{eq:ghPos} and \eqref{eq:ghx1Pos} follow from the
proof for \eqref{eq:gh}. For $B\in[\overline{B}_1,\infty)$,
$x_1(B)=0$. Thus, \eqref{eq:ghx1Pos2} follows from (\ref{eq:g1}).
\end{proof}
Recall the definition $\tilde g(B)$ in \eqref{eq:tildeg} of
Lemma~\ref{lem:B1}. This time $\tilde g(B)$ is well defined
for $B\in(0,\infty)$. Recall also the definition of
$\underline{A}(B)$ in \eqref{eq:AunderlineB} and
$\overline{A}(B)$ in \eqref{eq:AoverlineB}. We have the following
lemma that is analogous to Lemma \ref{lem:B1}.
\begin{lemma}
\label{lem:x_i nonnegative}
The function $\tilde {g}(B)$ is independent of $A$. It is  continuous
and strictly increasing on $B\in(0,\infty)$. Furthermore,
\begin{eqnarray*}
 \lim_{B\downarrow 0}\tilde{g}(B)=0 \quad \text{ and } \quad
 \lim_{B\uparrow \infty}\tilde{g}(B)=\infty. \label{eq:lim
   tria g nonnegative}
\end{eqnarray*}
Therefore there exists $\underline{B}_1\in (0,\infty)$ such
that  \eqref{eq:B1} holds. Furthermore, for $B\in(\underline{B}_1,\infty)$,
\eqref{eq:underlineAlessoverlineA} holds.
\end{lemma}

\begin{proof}
First, we prove $\tilde g$ is strictly increasing.
For $B\in (0,\overline{B}_1)$, the expression for
$\frac{d \tilde{g}(B)}{d B}$ is identical to the one in \eqref{eq:d
  wide(g)/dB}. For $B\in(\overline{B}_1,\infty)$
\begin{eqnarray}
\frac{d \tilde{g}(B)}{d B}
=-\frac{1}{\mu}h'(x_2(B))\frac{dx_2(B)}{dB}+e^{\lambda a}=
-e^{-\lambda(x_2(B)-a)}+e^{\lambda a} >0.
 \label{eq:d wide(g)/dB-2}
\end{eqnarray}
Thus, $\tilde g$ is strictly increasing.

Next we  prove  $\lim_{B\uparrow  \infty}\tilde{g}(B)=\infty$. We observe
that  (\ref{eq:d wide(g)/dB-2}) and \eqref{eq:x2BlimitPos}  imply that
$\lim_{B\uparrow -\infty}\frac{d \tilde{g}(B)}{d B}=e^{\lambda a}>0$,
from which we have  $\lim_{B\uparrow \infty}\tilde{g}(B)=\infty$.

The remaining proof of the lemma is identical to that of Lemma~\ref{lem:B1}.

\end{proof}

With Lemma \ref{lem:x_i nonnegative} replacing Lemma~\ref{lem:B1},
Lemmas~\ref{lem:Uu} and
\ref{lem:Lambda2continuity}  hold
without any modification.

\begin{lemma} \label{lem:B2}
The function $\Lambda_2(\overline{A}(B), B)$ is continuous and
strictly increasing in  $B\in (\underline{B}_1,\infty)$.
Furthermore,
\begin{eqnarray}
  \label{eq:Lambda2overlineA1}
&&\lim_{B\downarrow \underline{B}_1}  \Lambda_2(\overline{A}(B), B)=0, \\
&& \lim_{B\uparrow \infty}  \Lambda_2(\overline{A}(B), B)=\infty.
  \label{eq:Lambda2overlineA}
\end{eqnarray}
Therefore,
 there exists a unique
$\underline{B}_2\in (\underline{B}_1,\infty)$ such that (\ref{eq:B2}) holds.
\end{lemma}
\begin{proof}
The proof of this lemma is identical to the proof of
Lemma~\ref{lem:B1B2} except that we need to prove
\eqref{eq:Lambda2overlineA}.

To prove \eqref{eq:Lambda2overlineA}, we follow the expression in
\eqref{eq:dLamma/da}
for $\frac{\partial \Lambda_2(\overline{A}(B),B)}{\partial B}$.
By \eqref{eq:dU/dA} and \eqref{eq:du/dA}, we know that
$ U(\overline{A}(B),B)$ decreases in $B$ and $
u(\overline{A}(B),B)$ increases in $B$. Also we know that $x_1(B)=0$
for $B\in (\overline{B}_1,\infty)$. Following
the expression in
\eqref{eq:dLamma/da} and these facts, we have
\begin{displaymath}
 \lim_{B\uparrow\infty}\frac{\partial
   \Lambda_2(\overline{A}(B),B)}{\partial B}>0,
\end{displaymath}
which implies that \eqref{eq:Lambda2overlineA}.
\end{proof}

Lemma \ref{lem:Lambda2continuity} and Lemma \ref{lem:B2} immediately
gives the following lemma.
\begin{lemma}
  \label{lem:AstarPos}
For each $B\in [\underline{B}_2,\infty)$, there exists a
unique $A^*(B)\in (\underline{A}(B), \overline{A}(B)]$ such that
\eqref{eq:Astar} holds.
\end{lemma}

Finally, the following lemma gives a proof of Theorem
\ref{lem:optimalParameter nonnegative}.
\begin{lemma}
\label{lem:d D nonnegative}
There exist unique $B^*$ with $B^*\in (\underline{B}_2,\infty)$, $D^*$, $d^*$ and $\alpha^*$ such that
\begin{eqnarray*}
&&g_{A^*(B^*),B^*}(D^*))=-k,\\
&&g_{A^*(B^*),B^*}(d^*)=-k-\alpha^*,\\
&&\int_{d^*}^{D^*}\Big[g_{A^*(B^*),B^*}(x)+k\Big]dx=-K,\\
&&\alpha^* d^*=0,\ \ and \\
&& d^*\geq 0.
\end{eqnarray*}
where $\alpha^*\geq 0$.
\end{lemma}

\begin{proof}
For any $B\in (\underline{B}_2,\infty)$, $A^*(B)<  \overline{A}(B)$.
Therefore, \eqref{eq:gless-k} implies that
$g_{A^*(B),B}(x_1(B))< -k$.
Thus, there exists a unique $D(B)$ such that
\begin{equation}
  \label{eq:DBPos}
D(B)>0,\quad g_{A^*(B),B}(D(B))=-k,\quad g_{A^*(B),B}'(D(B))>0.
\end{equation}

If $x_1(B)>0$ and  $g_{A^*(B),B}(0)> -k$, then there exists a unqiue
 $d(B)$ such that
 \begin{equation}
   \label{eq:dBPos}
d(B)\ge 0,\quad g_{A^*(B),B}(d(B))=-k,\quad g_{A^*(B),B}'(d(B))<0.
 \end{equation}
Inequality \eqref{eq:dx_1/dB} shows that $x_1(B)$ is strictly decreasing
in $B\in (0, \overline{B})$. Also for $B\in (\underline{B}_2,\infty)$, \eqref{eq:dA/dB} implies that
\begin{eqnarray}
\frac{\partial g_{A^*(B),B}(0)}{\partial B}
=\frac{\int_{U(B)}^{u(B)}\Big[e^{-\lambda(x-a)}-e^{\lambda a}\Big]dx}{u(B)-U(B)}
<0.\label{eq:dg(0)/dB nonnegative}
\end{eqnarray}
Therefore, $g_{A^*(B),B}(0)$ is strictly decreasing in $B\in (\underline{B}_2,\infty)$.
Let $(\underline{B}_2, \underline{B}_4)$
be the interval over which $g_{A^*(B),B}(0)> -k$. If there is no $B$ that satisfies
$g_{A^*(B),B}(0)> -k$, then set $\underline{B}_4=\underline{B}_2$.
Thus, for $B\in (\underline{B}_2,\overline{B}_1\wedge \underline{B}_4)$,
$d(B)>0$. Otherwise,
for $B\in [\overline{B}_1\wedge \underline{B}_4,\infty)$,
 we set $d(B)=0$. The rest of the proof mimics the proof of
Lemma \ref{lem:d D}.

Define $\Lambda_1(A^*(B), B)$  as in (\ref{eq:Lambda1Astar}).
We are going to prove that
$\Lambda_1(A^*(B),B)$ is continuous and strictly decreasing in $B\in[\underline{B}_2,\infty)$
and
\begin{equation*}
  \lim_{B\downarrow \underline{B}_2} \Lambda_1(A^*(B),B)=0 \quad \text{and}
  \quad
  \lim_{B\uparrow \infty}\Lambda_1(A^*(B),B)=-\infty.
\end{equation*}
Therefore,
there exists a unique $B^*\in(\underline{B}_2,\infty)$,
such that 
\begin{eqnarray*}
\Lambda_1(A^*(B^*),B^*)=-K,
\end{eqnarray*}
from which one proves the lemma by choosing
 $A^*=A^*(B^*)$, $D^*=D(B^*)$, $d^*=d(B^*)$ and
$\alpha^*=(k+g_{A^*(B^*),B^*}(0))^-$.

We first show that $\Lambda_1(A^*(B), B)$  is continuous and strictly
decreasing in $B\in[\underline{B}_2,\infty)$. Observe that
\eqref{eq:dA/dB} continues  to hold, from which \eqref{eq:dgAstardB}
continues to hold.  We now claim that \eqref{eq:dLambda1AstardB}
continues to hold for $B\in (\underline{B}_2,\infty)$ except possibly
at $\overline{B}_1\wedge\underline{B}_4$.
Indeed, for $B\in (\underline{B}_2,\overline{B}_1\wedge\underline{B}_4)$
 \eqref{eq:dLambda1AstardB} holds as before.
For $B\in (\overline{B}_1\wedge\underline{B}_4,\infty)$,
$d(B)=0$ and  \eqref{eq:dLambda1AstardB} holds as well in this case.
This proves that $\Lambda_1(A^*(B), B)$ is continuous and decreasing
in $B$.

Next, it is easy to see that the limit
\eqref{eq:Lambda1AstarupperLimit} continues to hold as well.
It remains to prove
\begin{equation}
  \label{eq:Lambda1AstarlowerLimit}
 \lim_{B\uparrow \infty}  \Lambda_1(A^*(B), B) = -\infty.
\end{equation}
We will  prove next that
\begin{equation}
\label{eq:Lambda1Astarineq}
\lim_{B\uparrow \infty}\frac{d
  \Lambda_1(A^*(B),B)}{d B}<0,
\end{equation}
from which (\ref{eq:Lambda1AstarlowerLimit}) immediately follows.

To see (\ref{eq:Lambda1Astarineq}), using \eqref{eq:DBPos}, we have
\begin{eqnarray}
 \frac{d D(B)}{d B}
 &=&\frac{e^{-\lambda
     (D(B)-a)}-\frac{\int_{U(A^*(B),B)}^{u(A^*(B),B)}e^{-\lambda(x-a)}dx}{u(A^*(B),B)-U(A^*(B),B)}}{g_{A^*(B),B}'(D(B))}
 \label{eq:dDdBless0}\\
 &=&\frac{\int_{U(A^*(B),B)}^{u(A^*(B),B)}\Big[e^{-\lambda (D(B)-a)}-e^{-\lambda(x-a)}\Big]dx}{(u(A^*(B),B)-U(A^*(B),B)g_{A^*(B),B}'(D(B))} \nonumber\\
 &>&0. \nonumber
\end{eqnarray}
To study the limit in (\ref{eq:Lambda1Astarineq}),
we only need to consider $\Lambda_1(A^*(B),B)$  for $B\in (\overline{B}_1\wedge\underline{B}_4,\infty)$.
When $B\in (\overline{B}_1\wedge\underline{B}_4,\infty)$,
 $d(B)=0$ and hence  $\Lambda_1(A^*(B),B)
 =\int_{0}^{D(B)}\bigl[g_{A^*(B),B}(x)+k\bigr]dx$.
Therefore, for $B\in (\overline{B}_1\wedge\underline{B}_4,\infty)$,
\begin{eqnarray*}
&&\frac{d \Lambda_1(A^*(B),B)}{d B}\\
&=& \int_{0}^{D(B)} \Bigl[\frac{d A^*(B)}{d B}-e^{-\lambda(x-a)}\Bigr]dx+\frac{d D(B)}{d B} \bigl[g_{A^*(B),B}(D(B))+k\bigr]\\
&=&\int_{0}^{D(B)} \Bigl[\frac{d A^*(B)}{d B}-e^{-\lambda(x-a)}\Bigr]dx,
\end{eqnarray*}
where the last equality is due to $g_{A^*(B),B}(D(B))=-k$.
It follows from (\ref{eq:dA/dB})  that
\begin{eqnarray*}
\frac{d \Lambda_1(A^*(B),B)}{d B}
&=& \int_{0}^{D(B)} \Big[\frac{\int_{U(B)}^{u(B)}e^{-\lambda(y-a)}dy}{u(B)-U(B)}-e^{-\lambda(x-a)}\Big]dx\\
&\leq& \int_{0}^{D(B)} \Big[\frac{\int_{U(B)}^{u(B)}e^{-\lambda(D(B)-a)}dy}{u(B)-U(B)}-e^{-\lambda(x-a)}\Big]dx\\
&=& \int_{0}^{D(B)} [e^{-\lambda(D(B)-a)}-e^{-\lambda(x-a)}]dx,
\end{eqnarray*}
where the inequality follows from $D(B)<U(B)$. Inequality
(\ref{eq:dDdBless0})  implies that
\begin{eqnarray*}
\frac{d \Big(\int_{0}^{D(B)} [e^{-\lambda(D(B)-a)}-e^{-\lambda(x-a)}]dx\Big) }{d B}
=-\lambda e^{-\lambda(D(B)-a)} D(B)\frac{d D(B)}{d B}<0.
\end{eqnarray*}
For any $B_0\in(\overline{B}_1\wedge\underline{B}_4,\infty)$, we
have $\int_{0}^{D(B_0)}
[e^{-\lambda(x-a)}-e^{-\lambda(D(B_0)-a)}]dx<0$. Therefore,
\begin{eqnarray*}
\lim_{B\uparrow \infty}\frac{d \Lambda_1(A^*(B),B)}{d B}
&\leq& \lim_{B\uparrow \infty}\int_{0}^{D(B)}
[e^{-\lambda(x-a)}-e^{-\lambda(D(B)-a)}dy]dx \\
&\leq & \int_{0}^{D(B_0)}[e^{-\lambda(x-a)}-e^{-\lambda(D(B_0)-a)}]dx<0,
\end{eqnarray*}
proving
\eqref{eq:Lambda1Astarineq}.
\end{proof}

\section{Concluding Remarks}
\label{sec:conclusions}
In this paper, we have given a tutorial of the lower-bound approach to
studying the optimal control of Brownian inventory models with a
general convex holding cost function. The control can be either
impulse or singular, and the inventory can be either backlogged or
without backlog. For  future research, it would be interesting
to study multi-stage inventory systems with Brownian motion demand.  Yao
\cite{Yao10} has done a preliminary study for these systems.

\section*{Acknowledgments}
The authors thank Hanqin Zhang of Chinese Academy of Sciences and
National University of Singapore for stimulating discussions.  Part of this
work was done while the second author was visiting the School of
Industrial and Systems Engineering at Georgia Institute of Technology.
He would like to thank the hospitality of the school.

\bibliography{dai}

\end{document}